\documentclass[final]{siamltex}

\usepackage{amsmath}              
\usepackage{graphicx}
\usepackage{amssymb}
\usepackage{hyperref}
\usepackage{tikz} 
\usetikzlibrary{arrows}

\newtheorem{prop}{Proposition}[section]
\newtheorem{remark}{Remark}[section]

\newcommand{\vectornorm}[1]{\parallel\hspace{-.1cm}#1\hspace{-.1cm}\parallel}

\newenvironment{myindentpar}[1]
{\begin{list}{}
         {\setlength{\leftmargin}{#1}}
         \item[]
}
{\end{list}}

\title{Persistence and permanence of mass-action and power-law dynamical systems}

\author{Gheorghe Craciun\thanks{Department of Mathematics and Department of Biomolecular 			Chemistry, University of Wisconsin-Madison ({\tt craciun@math.wisc.edu}).} 
	\and Fedor Nazarov\thanks{Department of Mathematics, University of Wisconsin-Madison ({\tt 		nazarov@math.wisc.edu}).}
         \and Casian Pantea\thanks{Department of Mathematics and Department of Biomolecular 			Chemistry, University of Wisconsin-Madison ({\tt pantea@math.wisc.edu}).}}

\begin{document}

\maketitle


\begin{abstract}
Persistence and permanence are properties of dynamical systems that describe the long-term behavior of the solutions, and in particular specify whether positive solutions approach the boundary of the positive orthant. Mass-action systems (or more generally power-law systems) are very common in chemistry, biology, and engineering, and are often used to describe the dynamics in interaction networks. We prove that two-species mass-action systems derived from weakly reversible networks are both persistent and permanent, for any values of the reaction rate parameters. Moreover, we prove that a larger class of networks, called endotactic networks, also give rise to permanent systems, even if we allow the reaction rate parameters to vary in time. These results also apply to power-law systems and other nonlinear dynamical systems. In addition, ideas behind these results allow us to prove the Global Attractor Conjecture for three-species systems.
\end{abstract}


\pagestyle{myheadings}
\thispagestyle{plain}

\section{Introduction} 
Determining qualitative properties of solutions of dynamical systems arising from nonlinear interactions is generally a daunting task. However, a relevant mathematical theory that pertains to biochemical interactions obeying mass-action kinetics has been developed over the last 40 years, starting with the seminal work of  Fritz Horn, Roy Jackson and Martin Feinberg, \cite{cinci, patru, paispe}. Generally termed ``Chemical Reaction Network Theory", this theory establishes qualitative results that describe the surprisingly stable dynamic behavior of large classes of mass-action systems, {\em independently of the values of the reaction rate parameters} \cite{cinci}. This fact is especially useful since the exact values of the system parameters are usually unknown. 

Here we focus on the properties of {\it persistence and permanence} for mass-action systems, and the more general power-law systems. A dynamical system on $\mathbb{R}_{>0}^n$ is called persistent if no trajectory that starts in the positive orthant has an $\omega$-limit point on the boundary of $\mathbb{R}_{>0}^n$, and is called permanent if all trajectories that start in the positive orthant eventually enter a compact subset of $\mathbb{R}_{>0}^n$. Persistence and permanence are important in understanding properties of biochemical networks (e.g., will each chemical species be available indefinitely in the future), and also in ecology (e.g., will a species become extinct in an ecosystem), and in the dynamics of infectious diseases (e.g., will an infection die off, or will it infect the whole population).

In the context of biochemical networks we formulate the following conjecture:

\smallskip
\noindent
\textbf{\textsc{Persistence Conjecture. }}{\em  Any weakly reversible mass-action system is persistent.}
\smallskip

\noindent Here a {\em weakly reversible} mass-action system is one for which its directed reaction graph has the property that each of its connected components is strongly connected\footnote{See Definition \ref{def:WR}.}  \cite{cinci}. 
A version of this conjecture was first mentioned in \cite{cincijumate}\footnote{The version in \cite{cincijumate} only required that no positive trajectory converges to a boundary point; we conjecture that even more is true: no positive trajectory has an $\omega$-limit point on the boundary.}.  

In this paper we prove the Persistence Conjecture for two-species networks. Moreover, we introduce a new class of networks called endotactic networks, which contains the class of weakly reversible networks, and we show that  {\em endotactic two-species  $\kappa$-variable mass action systems are permanent} (Theorem \ref{thm:perm}). 

For two-species systems, this is a stronger statement than the Persistence Conjecture, in several ways. First, {\em permanence} here means that there is a compact set $K$ in the interior of the positive orthant such that all forward trajectories of positive initial condition end up in $K$. Therefore, our result implies both strong persistence and uniform boundedness of trajectories for two-species $\kappa$-variable mass-action systems. Second, our mass-action system is $\kappa$-{\em variable}, meaning that the reaction rate constants are allowed to vary within a compact subset of $(0,\infty)$. This makes our result applicable to a larger class of kinetics than just mass-action. Finally, the class of {\em endotactic} networks, explained in detail in section \ref{section:endo} strictly contains the class of weakly reversible networks. We conjecture that endotactic  $\kappa$-variable mass action systems are permanent for any number of species. 

On the other hand, note that the Persistence Conjecture for three or more species remains open.

The Persistence Conjecture is also strongly related to the Global Attractor Conjecture, which was first stated over 35 years ago \cite{patru}, and is one of the main open questions in Chemical Reaction Network Theory. It is concerned with the global asymptotic stability of positive equilibria for the class of ``complex-balanced" \cite{paispe,patru,cinci,cincijumate} mass-action systems. It is known that such systems admit a unique positive equilibrium $c_{\Sigma}$ within each stoichiometric compatibility class\footnote{A stoichiometric compatibility class is a minimal linear invariant subset; see Definition \ref{def:compClass}.}  $\Sigma$. Moreover, each such equilibrium  admits a strict Lyapunov function and therefore $c_{\Sigma}$ is locally asymptotically  stable with respect to ${\Sigma}$ \cite{cinci, paispe}.
However, the existence of this Lyapunov function does not guarantee that $c_{\Sigma}$ is a global attractor, which is the object of the following conjecture:
\smallskip

\noindent
\textbf{\textsc{Global Attractor Conjecture.}} {\em Given a complex-balanced mass-action system and any of its stoichiometric compatibility classes ${\Sigma}$, the positive equilibrium point $c_{\Sigma}$ is a global attractor on $int({\Sigma}).$}
\smallskip

Trajectories of complex-balanced mass-action systems must converge to the set of equilibria  \cite{sase}. This makes the Global Attractor Conjecture equivalent to showing that any complex-balanced system is {\em persistent}, i.e. no trajectory with positive initial condition gets arbitrarily close to the boundary of the positive orthant. Therefore, since a complex-balanced mass-action system is necessarily weakly reversible \cite{cinci}, a proof of the Persistence Conjecture would imply the Global Attractor Conjecture. In general, both conjectures are still open; the Persistence Conjecture has not been proved previously in any dimension, but recent  developments have been achieved towards a proof of the Global Attractor Conjecture. For instance, Anderson \cite{sapte}  and Craciun, Dickenstein, Shiu and Sturmfels \cite{opt} showed that vertices of the stoichiometric compatibility class $\Sigma$ cannot be $\omega$-limit points and that the Global Attractor Conjecture is true for a class of systems with two-dimensional stoichiometric compatibility classes. Anderson and Shiu proved that for a weakly reversible mass-action system, the trajectories that originate in the interior of $\Sigma$ are ``repelled"  away from the codimension-one faces of $\Sigma$. This allowed them to prove that the Global Attractor Conjecture is true if the stoichiometric compatibility class is two-dimensional \cite{noua}.  

\section{Definitions and notation}  

A chemical reaction network is usually given by a finite list of reactions that involve a finite set of chemical species. For example, a reaction network involving two species $A_{1}$ and $A_{2}$  is given in Figure \ref{Fig:A1A2}.

\begin{figure}[h]
\begin{center}
\includegraphics[width=5cm]{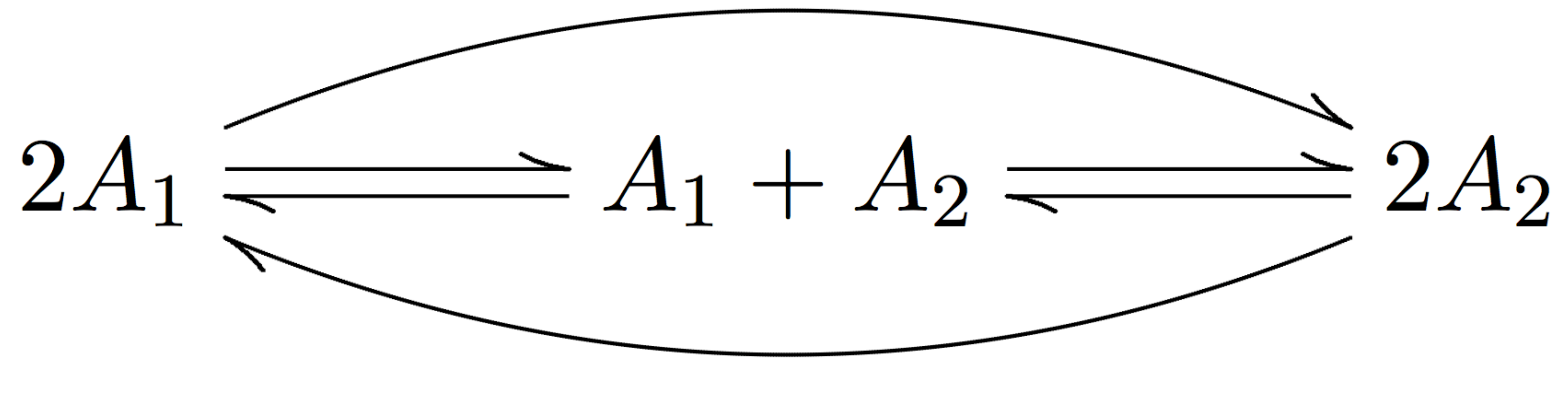}
\caption[Short]{A two-species reaction network.}\label{Fig:A1A2}
\end{center}
\end{figure}

We define the functions  $c_{A_{1}}(t)$ and $c_{A_{2}}(t)$ to be the molar concentrations of species $A_1, A_2$ at time $t$. 
The changes in the concentrations are dictated by the chemical reactions in the network; for instance, whenever the reaction $A_{1}+A_{2}\to 2A_{1}$ occurs, the net gain is a molecule of $A_{1}$, whereas one molecule of $A_{2}$ is lost.

\subsection{Chemical reaction networks} Here we recall some standard terminology of Chemical Reaction Network Theory (see \cite{cinci,paispe}). In what follows the set of nonnegative, respectively strictly positive real numbers are denoted by $\mathbb R_{\geq 0}$ and $\mathbb R_{>0}$ . For any integer $n>1$ we call $\mathbb R_{>0}^n$ the {\em positive orthant} and  $\mathbb R_{\geq 0}^n$ the {\em closed positive orthant.}
For an arbitrary finite set $I$ we denote by ${\mathbb Z}_{\geq 0}^I$ and ${\mathbb R}_{\geq 0}^I$  the set of all formal 
sums $\alpha=\displaystyle\sum_{i \in I}\alpha_ii$ where $\alpha_i$ are nonnegative integers, respectively reals. 
The \textit{support} of an element $\alpha\in \mathbb R^{I}$ is $supp(\alpha)=\{i\in I : \alpha_{i}\neq 0\}.$

\begin{definition}\label{def:CRN} 
A chemical reaction network is a triple $(\mathcal{S}, \mathcal{C}, \mathcal{R})$, where $\mathcal{S}$ is the set of chemical \textit{species},  $\mathcal{C}\subseteq \mathbb Z_{\geq 0}^{\mathcal S}$ is the set of \textit{complexes} (i.e., the objects on left or right side of the reaction arrows), and $\mathcal{R}$ is a relation on $\mathcal{C}$, denoted $P\to P'$, representing the set of \textit{reactions} in the network. Here $P$ is called the source complex and $P'$ is called the target complex of the reaction $P\to P'.$ Moreover, the set $\mathcal{R}$ must satisfy the following three conditions: it cannot contain elements of the form $P\to P$; for any $P\in \mathcal{C}$ there exists some $P'\in\mathcal{C}$ such that either $P\to P'$ or $P'\to P$; and the union of the supports of all $P\in\mathcal{C}$ is $\mathcal{S}$. 
\end{definition}

\noindent The second condition in Definition \ref{def:CRN} guarantees that each complex appears in at least one reaction, and the third condition assures that each species appears in at least one complex. For the reaction network in Figure \ref{Fig:A1A2}, the set of species is $\mathcal{S}=\{A_{1}, A_{2}\}$, the set of complexes is $\mathcal{C}=\{2A_{1}, A_{1}+A_{2}, 2A_{2} \}$ and the set of reactions is $\mathcal{R}=\{2A_{1}\rightleftharpoons A_{1}+A_{2}, A_{1}+A_{2}\rightleftharpoons 2A_{2}, 2A_{2}\rightleftharpoons 2A_{1}\}$, which consists  of 6 reactions, represented as three reversible reactions.  For convenience, we refer to a chemical reaction network by specifying $\mathcal{R}$ only, since it encompasses all the information about the network.

By a slight abuse of notation we may view all complexes as (column) vectors of dimension equal to the number of elements of $\mathcal S$, via an identification given by a fixed ordering of the species. To exemplify, the complexes in Figure \ref{Fig:A1A2} are $2A_{1}=\begin{bmatrix} 2 \\ 0 \\\end{bmatrix}$, $A_{1}+A_{2}=\begin{bmatrix} 1 \\ 1 \\\end{bmatrix}$, and $2A_{2}=\begin{bmatrix} 0 \\ 2 \\\end{bmatrix}$. 
For any reaction $P\to P'$ we may thus consider the vector $P'-P,$ called the {\em reaction vector} of $P\to P'.$ 
\begin{definition}\label{def:compClass}
The stoichiometric subspace $S$ of $\cal R$ is defined as 
$$S= span\{P'-P\mid P\to P'\in \cal R\}.$$
\end{definition}
\indent Suppose $\cal R$ has $d$ species, fix an order among them and $c(t)\in \mathbb R^{\cal S}\equiv\mathbb R^{d}$  denote the vector of molar species concentrations at time $t$.  We refer to $c(t)$ as the {\em concentration vector} of $\cal R$ at $t$. \noindent As we will see soon, for all $t\ge 0,$ the concentration vector $c(t)$ is constrained to the {\em stoichiometric compatibility class} of $c(0),$ a special polytope in $\mathbb R_{\ge 0}^d,$ defined below. 

\begin{definition}
Let $c^0\in \mathbb R^d.$ The polytope $(c^0+S)\cap \mathbb R_{\ge 0}^d$ is called the stoichiometric compatibility class of $c^0.$
\end{definition}

The reaction network $\cal R$ can be viewed as a directed graph whose vertices are given by complexes and whose edges correspond to reactions of $\cal R$. The connected components of this graph are called {\em linkage classes} of $\cal R.$ 

\begin{definition}
A reaction network $\cal R$ is called reversible if $P'\to P\in\cal R$ whenever $P\to P'\in \cal R.$
\end{definition}  

\begin{definition}\label{def:WR} 
A reaction network $\cal R$ is called weakly reversible if its associated directed graph has strongly connected components.
\end{definition}

\subsection{$\kappa$-variable mass-action systems} Throughout this paper, the differential equations that govern the evolution of the concentration vector $c(t)$ will be given by {$\kappa$-variable mass-action} kinetics. The next definition clarifies this notion; to state it, we need the following notation: given two vectors $u=\displaystyle\sum_{s\in\mathcal{S}}u_{s}s$ and $v=\displaystyle\sum_{s\in\mathcal{S}}v_{s}s$ in $\mathbb{R}^{\mathcal{S}}_{\ge 0}$, we denote $u^{v}=\displaystyle\prod_{s\in\mathcal{S}}(u_{s})^{v_{s}}$, with the convention $0^{0}=1$.
\begin{definition}\label{def:massAct}
A  $\kappa$-variable mass-action system is a quadruple $({\cal S}, {\cal C}, {\cal R}, \kappa)$ where $({\cal S}, {\cal C}, {\cal R})$ is a reaction network and $\kappa:\mathbb R_{\geq 0}\to (\eta, 1/\eta)^{\cal R}$ for some $\eta < 1$ is a piecewise differentiable function called the rate constants function.  Given the initial condition $c(0)=c^0\in \mathbb R_{\geq 0}^d,$ the concentration vector $c(t)$ is the solution of the $\kappa$-variable mass-action ODE system 
\begin{equation}\label{varMassAct}
\dot c(t) = \sum_{P\to P'}\kappa_{P\to P'}(t) c(t)^P (P'-P).
\end{equation}
\end{definition} 
\indent The $\kappa$-variable mass-action kinetics is a natural generalization of the widely used {\em mass-action kinetics}, where the rate constant functions $\kappa(\cdot)$ are simply positive constants. Therefore any results we prove for $\kappa$-variable mass-action systems are also true for mass-action systems.

Note that integration of (\ref{varMassAct}) yields 
$$c(t) = c^0 + \sum_{P\to P'} \left(\int_{0}^t \kappa_{P\to P'}(s)c(s)^P ds\right) (P'-P),$$
and therefore $c(t)\in c(0)+S$ for all times $t\geq 0.$ Moreover, it is easy to check that $c(t)\in \mathbb R^d_{\geq 0};$ indeed, all the negative terms in the expression of $\dot c_i(t)$ coming from the right hand side of (\ref{varMassAct})  contain $c_i(t)$ as a factor and vanish on the face $c_i=0$ of $\mathbb R_{\ge 0}^d.$ We conclude that  $(c^0+S)\cap \mathbb R_{\ge 0}^d$ is forward invariant for (\ref{varMassAct}). Similarly one can notice that $(c^0+S)\cap \mathbb R_{>0}^d$ is also forward invariant for (\ref{varMassAct}).

\subsection{Persistence and permanence of dynamical systems} 
We review some more vocabulary (\cite{Freedman, Schuster, LV}).

\begin{definition}\label{def:pers} A d-dimensional dynamical system is called persistent if for any forward trajectory $T(c_0)=\{c(t)=(x_1(t),\ldots, x_d(t))\mid t\geq 0\}$ with positive initial condition $c_0\in\mathbb R_{>0}^d$ we have
$$\liminf_{t\to \infty} x_i(t)>0\ for\ all\ i\in\{1,\ldots, d\}.$$
\end{definition}
\indent Note that some authors call a dynamical system that satisfies the condition in Definition \ref{def:pers} {\em strongly persistent} \cite{LV}. In  their work, persistence requires only that  $\limsup_{t\to \infty} x_i(t)>0$ for all $i\in\{1,\ldots, d\}.$ 

Going back to Definition \ref{def:pers}, persistence means that no forward trajectory with positive initial condition approaches the coordinate axes arbitrarily close. Note that  a dynamical system with bounded trajectories is persistent if the open positive quadrant is forward invariant and there are no $\omega$-limit points on $\partial\mathbb R_{\ge 0}^d;$  recall the definition of $\omega$-limit points below.

\begin{definition}
Let $T(c_0)=\{c(t)\mid t\geq 0\}$ denote a forward trajectory of a dynamical system with initial condition $c_0\in \mathbb R_{>0}^d.$ The {\em $\omega$-limit set of $T(c_0)$} is
$$\textstyle\lim_{\omega} T(c_0) = \{l\in \mathbb R^d\mid \lim_{n\to\infty}c(t_n) = l \text{ for some sequence } t_n\to\infty\}.$$
The elements of $\lim_{\omega} T(c_0)$ are called $\omega$-limit points of $T(c_0)$.
\end{definition}
\begin{definition}
A d-dimensional dynamical system is called permanent on a forward invariant set $D\subseteq \mathbb R_{\ge 0}^d$ if there exists $\epsilon > 0$ such that for any forward trajectory $T(c_0)$ with positive initial condition $c_0\in D$ we have
$$\epsilon < \liminf_{t\to \infty} x_i(t)\ and\ \limsup_{t\to \infty} x_i(t)<1/\epsilon\ for\ all\ i\in\{1,\ldots, d\}.$$
\end{definition} 
\indent In other words, a dynamical system is permanent on a forward invariant set $D\subseteq \mathbb R^d_{\ge 0}$ if there is a compact set $K\subset \mathbb R_{> 0}^d$ such that any forward trajectory with positive initial condition $c_0\in D$ ends up in $K$. 

In the case of a $\kappa$-variable mass-action system, where the trajectories are confined to subspaces of $\mathbb R^d,$ it is meaningful to require that the system is permanent {\em on each stoichiometric compatibility class}, as we do in the following definition.

\begin{definition}
A $\kappa$-variable mass-action system is called permanent if, for any stoichiometric compatibility class $\Sigma,$ the dynamical system is permanent on $\Sigma$.
\end{definition}

\noindent Clearly, a permanent $\kappa$-variable mass-action system is persistent.

\section{An illustrative example}\label{sec:example}
Before introducing endotactic networks and proving their persistence (Theorem \ref{thm:main}), we take an intermediary step and discuss an example which illustrates some key points in our strategy for the proof of Theorem \ref{thm:main}. Consider the following $\kappa$-variable mass-action system:
\begin{equation}\label{ex:Illustrative}
2X\displaystyle \mathop{\rightleftharpoons}^{k_1}_{k_{-1}}Y\quad
X\displaystyle \mathop{\rightleftharpoons}^{k_2}_{k_{-2}}Y\quad
X\displaystyle \mathop{\rightleftharpoons}^{k_3}_{k_{-3}}2X+Y
\end{equation}
\noindent where $k_i=k_i(t)\in (\eta,1/\eta)$ for all $i\in\{-3,-2,\ldots, 3\}$ and all $t\ge 0.$  Let $c(t)=(x(t),y(t)),$ where $x(t)$ and $y(t)$ denote the concentrations of species $X$ and $Y$ at time $t\ge 0.$ Then $(x(t),y(t))$ satisfy the system of differential equations (\ref{ex:ODE}):
\begin{equation}\label{ex:ODE}
\begin{bmatrix} \dot x(t) \\ \dot y(t) \\\end{bmatrix} = 
[k_1x(t)^2 - k_{-1}y(t)]\begin{bmatrix} -2 \\ 1 \\\end{bmatrix}+
[k_2x(t) - k_{-2}y(t)] \begin{bmatrix} -1 \\ 1 \\\end{bmatrix} +
[k_3x(t) - k_{-3}x(t)^2y(t)]\begin{bmatrix} 1 \\ 1 \\\end{bmatrix}
\end{equation}
Let us fix an initial condition $c(0)$ for (\ref{ex:ODE}). We want to construct a convex polygon $\cal P$ such that $conv(\cal P)$ is forward invariant for the dynamics given by (\ref{ex:ODE}) and contains $c(0),$ where $conv(\cal P)$ is the union of $\cal P$ and its interior\footnote{In this case we will simply say that $\cal P$ is forward invariant for (\ref{ex:ODE}).}. According to a theorem of Nagumo (see \cite{nagumo,Blanchini}), this is equivalent to requiring that the flow satisfy the {\em sub-tangentiality condition}, which in our case simply means that the vector $\dot c(t_0)$ points towards the interior of $\cal P$  whenever $c(t_0)\in\cal P.$ 
In example (\ref{ex:Illustrative}), the flow has three components along reaction vectors corresponding to the three pairs of reversible reactions (see (\ref{ex:ODE})). We would like to construct $\cal P$ such that {\em each component of the flow satisfies the sub-tangentiality condition}; this will clearly imply that the aggregate flow (\ref{ex:ODE}) also satisfies the sub-tangentiality condition and therefore $\cal P$ is forward invariant for (\ref{ex:ODE}). 

\begin{figure}
\begin{center}
\begin{tabular}{cc}
%
\begin{tikzpicture}[>=stealth, scale =.34]
\node at (-1,15) {(a)};
\begin{scope}
\clip (0,0) rectangle+(16,16);

\draw (0,0)--(16,0);
\draw (0,0)--(0,16);
\draw [gray](16,0)--(16,16);
\draw [gray](0,16)--(16,16);

\draw[dotted, thick, fill=blue!25!white] plot[domain = 0:25.1] (\x,.055*\x^2);
\draw[dotted, thick, fill = white] plot[domain = 0:16] (\x,.12*\x^2);

\draw[thick, ->] (.5, 15.5)--(1.5,15);
\draw[thick, ->] (15.5,.5)--(14.5,1);
\draw[thick, ->] (14,15)--(13,15.5);
\draw[thick, ->] (14,15)--(15,14.5);

\draw (4,.5)--(2,1.5)--(1,7)--(5,13)--(10,14.75)--(15.5,12) -- (14,3)--(4,.5);
\draw [blue, very thick](4,.5)--(2,1.5);
\draw [blue, very thick](15.5,12)--(10,14.75);

\fill[blue] (4,.5)circle (.15);
\fill[blue] (2,1.5) circle (.15);
\fill[black]  (1,7) circle (.15);
\fill[black]  (5,13) circle (.15);
\fill[blue]  (10,14.75)circle (.15);
\fill[blue]  (15.5,12) circle (.15);
\fill[black]  (14,3) circle (.15);

\node[rotate = 63, text = black] at (7.5,8) {$y=(1/\delta)x^2$};
\node[rotate = 52, text = black] at (11.8,7) {$y=\delta x^2$};

\draw [gray](0,16)--(16,16);

\end{scope}
\draw (0,0)--(16,0);
\draw (0,0)--(0,16);
\draw [gray](16,0)--(16,16);
\draw [gray](0,16)--(16,16);

\end{tikzpicture} 
&
%
\begin{tikzpicture}[>=stealth, scale =.34]
\node at (-1,15) {(b)};
\begin{scope}
\clip (0,0) rectangle+(16,16);

\draw (0,0)--(16,0);
\draw (0,0)--(0,16);
\draw [gray](16,0)--(16,16);
\draw [gray](0,16)--(16,16);

\draw[color=white, fill = blue!25!white] (0,0)--(16,12.8)--(16,16)--(13.333,16)--(0,0);

\draw[dotted, thick, fill=blue!25!white] plot[domain = 0:25.1] (\x,.8*\x);
\draw[dotted, thick, fill = white] plot[domain = 0:16] (\x,1.2*\x);

\draw[thick, ->] (.5, 15.5)--(1.2,14.8);
\draw[thick, ->] (15.5,.5)--(14.8,1.2);
\draw[thick, ->] (15,15)--(15.8,14.2);
\draw[thick, ->] (15,15)--(14.2,15.8);

\draw (3,.5)--(.5,3)--(2,14)--(11.5,15)--(15.5,11)--(14,4) -- (3,.5);
\draw [blue, very thick](3,.5)--(.5,3);
\draw [blue, very thick](15.5,11)--(11.5,15);

\fill[blue] (3,.5)circle (.15);
\fill[blue] (.5,3) circle (.15);
\fill[black]  (2,14) circle (.15);
\fill[blue]  (11.5,15) circle (.15);
\fill[blue]  (15.5,11)circle (.15);
\fill[black]  (14,4) circle (.15);
\node[rotate = 50, text = black] at (7.5,9.7) {$y=(1/\delta)x$};
\node[rotate = 40, text = black] at (9,6.5) {$y=\delta x$};

\draw [gray](0,16)--(16,16);

\end{scope}
\draw (0,0)--(16,0);
\draw (0,0)--(0,16);
\draw [gray](16,0)--(16,16);
\draw [gray](0,16)--(16,16);

\end{tikzpicture} 
\\
%
\begin{tikzpicture}[>=stealth, scale =.34];
\node at (-1,15) {(c)};
\begin{scope}
\clip (0,0) rectangle+(16,16);

\draw (0,0)--(16,0);
\draw (0,0)--(0,16);
\draw [gray](16,0)--(16,16);
\draw [gray](0,16)--(16,16);

\draw[dotted, thick, smooth, fill=blue!25!white] plot[domain = 0.4:16.5] (\x,25/\x);
\draw[dotted, thick, smooth, fill=white] plot[domain = 0.4:16.5] (\x,50/\x);

\draw[thick, ->] (.5,.5)--(1.3,1.3);
\draw[thick, ->] (15.5,15.5)--(14.7,14.7);
\draw[thick, ->] (6,6)--(6.8,6.8);
\draw[thick, ->] (6,6)--(5.2,5.2);

\draw (1,12.5)--(4,15.5)--(12,15)--(14,10)--(15.5,4)--(12.5,1)--(1,3)--(1,12.5);
\draw [blue, very thick](1,12.5)--(4,15.5);
\draw [blue, very thick](12.5,1)--(15.5,4);

\fill[blue] (1,12.5)circle (.15);
\fill[blue] (4,15.5) circle (.15);
\fill[black]  (12,15) circle (.15);
\fill[black]  (14,10) circle (.15);
\fill[blue]  (12.5,1)circle (.15);
\fill[blue]  (15.5,4) circle (.15);
\fill[black]  (1,3) circle (.15);

\node[rotate = -25, text = black] at (11.7,5) {$y=(1/\delta)x^{-1}$};
\node[rotate = -80, text = black] at (2.2,10) {$y=\delta x^{-1}$};

\draw [gray](0,16)--(16,16);

\end{scope}
\draw (0,0)--(16,0);
\draw (0,0)--(0,16);
\draw [gray](16,0)--(16,16);
\draw [gray](0,16)--(16,16);

\end{tikzpicture} 
&
\begin{tikzpicture}[>=stealth, scale =.34]
\node at (-1,15) {(d)};
\begin{scope}
\clip (0,0) rectangle+(16,16);

\draw (0,0)--(16,0);
\draw (0,0)--(0,16);
\draw [gray](16,0)--(16,16);
\draw [gray](0,16)--(16,16);
%

\draw [blue!25!white, fill = blue!25!white] (0,0) .. controls + (0:5cm) and + (250:9cm) .. (12.5,16) -- (14.5,16) .. controls + (242:9cm) and (0:7) ..(0,0);
\draw [blue!25!white, fill = blue!25!white] (0,0) -- (16,11) --(16,13.5)--(0,0); 
\draw [blue!25!white, fill = blue!25!white] (.2,16) .. controls + (275:15cm) and + (175:5cm) .. (16,.2)
-- (16,1.4) .. controls + (175:12cm) and + (277:8cm) .. (1,16);

\draw [thick, dotted] (0,0) .. controls + (0:5cm) and + (250:9cm) .. (12.5,16);
\draw [thick, dotted] (14.5,16) .. controls + (242:9cm) and (0:7) ..(0,0);
\draw [thick, dotted] (0,0) -- (16,11);
\draw [thick, dotted] (16,13.5)--(0,0); 
\draw [thick, dotted] (.2,16) .. controls + (275:15cm) and + (175:5cm) .. (16,.2);
\draw [thick, dotted] (16,1.4) .. controls + (175:12cm) and + (277:8cm) .. (1,16);

\draw [blue, very thick](.5,2.5)--(2,1)--(3.5,.25)--(14,.25)--(15.5,1.75)--(15.5,10)--(13.5, 12)--(10.5, 13)--(2,13)--(.5,11.5)--(.5,2.5);

\fill[blue] (.5,2.5)circle (.15);
\fill[blue] (2,1) circle (.15);
\fill[blue]  (3.5,.25) circle (.15);
\fill[blue]  (14,.25) circle (.15);
\fill[blue]  (15.5,1.75)circle (.15);
\fill[blue]  (15.5,10) circle (.15);
\fill[blue]  (13.5,12) circle (.15);
\fill[blue]  (10.5,13) circle (.15);
\fill[blue]  (2,13)circle (.15);
\fill[blue]  (.5,11.5) circle (.15);

\draw [gray](0,16)--(16,16);
\end{scope}

\draw (0,0)--(16,0);
\draw (0,0)--(0,16);
\draw [gray](16,0)--(16,16);
\draw [gray](0,16)--(16,16);

\end{tikzpicture} 
\end{tabular}
\end{center}
\caption{Invariant polygons for the reactions in example (\ref{ex:Illustrative}), taken separately ((a),(b),(c)) and together (d). Outside the shaded regions the direction of the flow component corresponding to each reaction is known, while inside each shaded region at least one such reaction is unknown.}\label{Fig:ex}
\end{figure}
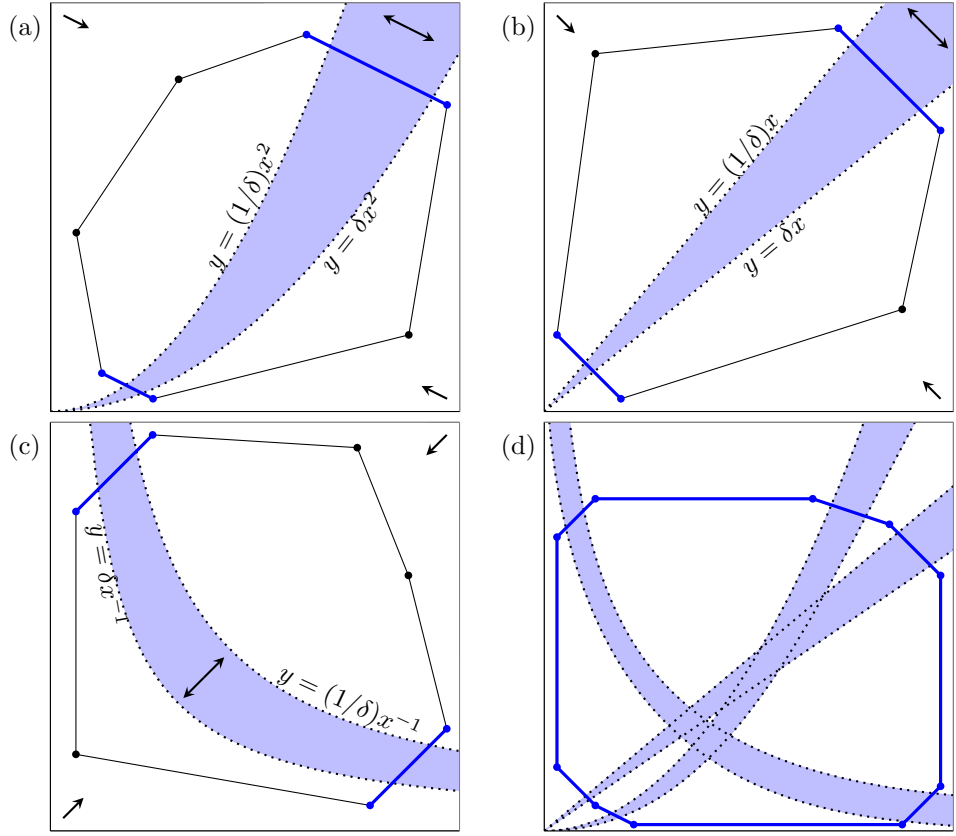

Following up on the preceeding observation, we next try to characterize the polygons that are forward invariant for one component of the flow, i.e., corresponding to a single reversible reaction. Let us fix $t$ for a moment. The first reaction of (\ref{ex:Illustrative}) drives the flow along the vector 
$\begin{bmatrix} -2 \\ 1 \\\end{bmatrix}$ with rate $k_1(t)x(t)^2 - k_{-1}(t)y(t).$ 
The sign of the rate gives the direction of the flow, 
$\begin{bmatrix} -2 \\ 1 \\\end{bmatrix}$ or $\begin{bmatrix} 2 \\ -1 \\\end{bmatrix},$ 
according to which side of the curve 
$y=(k_1(t)/k_{-1}(t))x^2$  the phase point $c(t)$ belongs to. 
The coefficient $k_1(t)/k_{-1}(t)$ is unknown, but we know it is bounded, 
$k_1(t)/k_{-1}(t)\in(\eta^2, 1/\eta^2).$ 
Therefore, (see Figure \ref{Fig:ex} (a)) the component of $\dot c(t)$ in (\ref{ex:ODE}) corresponding to the first reaction is in direction 
$\begin{bmatrix} 2 \\ -1 \\\end{bmatrix}$ if $c(t)\in \{(x,y)\in\mathbb R^2_{>0}\mid y > (1/\delta) x^2\}$ 
and in direction 
$\begin{bmatrix} -2 \\ 1 \\\end{bmatrix}$ if $c(t)\in \{(x,y)\in\mathbb R^2_{>0}\mid y < \delta x^2\},$ 
where $\delta = \eta^2.$ In the region between the curves above,
$$\Gamma = \{(x,y)\in\mathbb R^2_{>0}\mid (1/\delta) x^2 > y > \delta x^2\}$$
the flow corresponding to the first reaction can go both ways; the sub-tangentiality condition forces 
${\cal P}\cap \Gamma$ to be parallel to $\begin{bmatrix} -2 \\ 1 \\\end{bmatrix}.$ 
In other words, as in Figure \ref{Fig:ex}, 

\noindent (*) {\em $\Gamma$ cannot contain any vertices of  $\cal P$, and the two sides of $\cal P$ that intersect $\Gamma$ must  be parallel  to the reaction vector that corresponds to $\Gamma.$}

This is the only requirement that reaction $2X\displaystyle\mathop{\rightleftharpoons}^{k_1}_{k_{-1}}Y$ (or any other reversible reaction) imposes on $\cal P,$ since convexity  of $\cal P$ guarantees that sub-tangentiality is satisfied everywhere else.

Figure \ref{Fig:ex} (b), (c) illustrates the flow corresponding to the two remaining pairs of reversible reactions. We may now combine conditions (*) imposed by the three pairs of reversible reactions and obtain the polygon $\cal P$ in Figure \ref{Fig:ex} (d) that is forward invariant for (\ref{ex:ODE}).  We may realize $\cal P$ from a rectangle in the positive quadrant whose sides are parallel to to axes and whose corners are ``cut" at intersections with $\Gamma$-type regions in the direction of the corresponding reaction vector (see Figure \ref{Fig:ex} (d)). The original rectangle must be close enough to the axes and ``large" enough to contain $c(0)$ in its interior and to allow for cuts in regions where the curves $y = \delta x^{\sigma}, y = (1/\delta) x^{\sigma},$ $\sigma\in \{1, 2\}$ are ordered with respect to $\sigma.$ For instance, the cuts at the $(0,0)$ or SW corner are in a region where 
$$\delta x^2 < (1/\delta) x^2 < \delta x < (1/\delta) x,$$
whereas in the NE corner we have 
$$\delta x < (1/\delta) x < \delta x^2 < (1/\delta) x^2.$$
This arrangement of the curves avoids the ambiguity of cuts intersecting more than one $\Gamma$-type region  and is necessary for obtaining a convex $\cal P.$

We therefore conclude that the mass-action system (\ref{ex:Illustrative}) is persistent. In fact, our informal discussion can be made rigorous and extended to apply to any reversible network:
\begin{theorem}\label{thm:rev}
Any two-species reversible $\kappa$-variable mass-action system is persistent.
\end{theorem}
Although our discussion of Theorem \ref{thm:rev} has been informal, a rigorous proof will be given through the more general result of Theorem \ref{thm:main}.

In the next sections we extend ideas discussed here for the case of endotactic networks which we will define soon. For reversible networks, this section revealed an important ingredient in constructing $\cal P$ and proving persistence. Namely, when $c(t)\in\cal P,$ the aggregate effect of a pair of reversible reactions $P\rightleftharpoons P'$ is to push the trajectory towardss the corresponding $\Gamma$ region, therefore towardss the interior of $\cal P.$ The ``bad" reaction in $P\rightleftharpoons P',$ i.e. the one that points outside $\cal P$ is cancelled by the other, ``good reaction". Both reactions are good if $c(t)$ lies on one of the two sides parallel to $P'-P.$ Which reaction is good and which is bad comes from the sign of binomials of the form 
\begin{equation}\label{eq:binomials}
c^{P'}-C\cdot c^P,
\end{equation} 
\noindent where $C$ is some positive constant, in this case $C=\delta$ or $C=1/\delta.$ Analyzing the sign of  (\ref{eq:binomials}) for source complexes $P$ and $P'$ is a recurring theme throughout the paper. We will refer to this informally as ``comparing source monomials up to a constant."

\section{Endotactic networks}\label{section:endo} 
The results we obtained in this paper are applicable to  {\em endotactic networks,} a large class of two-dimensional reaction networks characterized by a simple geometric property. As we will see shortly, the class of endotactic networks is larger than the well-known class of weakly reversible networks. We build up to the definition of an endotactic network and begin with a few geometric notions.

Let $\cal R = (\cal S, \cal C, \cal R, \kappa)$ denote a $\kappa$-variable mass-action system with two species and let ${\cal S} = [X,\ Y]$ denote its ordered set of species.  

\begin{definition} By a slight abuse of terminology, the set of lattice points corresponding to sources of $\cal R:$
$${\cal{SC}}({\cal R}) = \{(m,n)\in \mathbb{Z}_{\geq0}^2\ such\ that\ mX+nY\in\cal C \text{ is a source complex} \}.$$
is also called the set of source complexes of $\cal R.$ The source monomial corresponding to $(m,n)\in{\cal{SC(R)}}$ is $x^my^n.$
\end{definition}
%
\begin{figure}[h]
\begin{center}
\begin{tikzpicture}[thick, >=stealth', scale = .55] 
\draw[->,thick]  (-4,3.5)--(-3,3.5);
\draw[->,thick]  (-4,3.5)--(-4.5,4.5);
\node at  (-3,3.8) {${\bf v}_1$};
\node at  (-4.8,4.4) {${\bf v}_2$};
\draw[gray] (-5,0) -- (3,4); 
\draw[<->, very thick] (-2,-1) -- (2,1); 
\draw[->, very thick] (-1,2) -- (0,4);
\draw[->, very thick] (-1,2) -- (-4,.5); 
\draw[draw = white, double=gray] (-1,-2) -- (-1,6); 
\shadedraw [shading=ball] (0,0) circle (.15cm);
\shadedraw [shading=ball] (-1,2) circle (.15cm); 
\shadedraw [] (2,1)+(.16,.09) circle (.15cm);
\shadedraw [] (-2,-1)+(-.16,-.09) circle (.15cm);
\shadedraw [] (-4,.5)+(-.16,-.09) circle (.15cm);
\shadedraw [] (0,4)+(.1,.16) circle (.15cm);
\node[rotate=28] at (2.3,4) {$esupp^{{\bf v}_2}(\cal R)$};
\node[rotate=90] at (-1.3,5.2) {$esupp^{{\bf v}_1}(\cal R)$};
\end{tikzpicture} 
\end{center}
\caption{Essential supports corresponding to vectors $\bf v_1$ and $\bf v_{2}$}\label{Fig:esupp}
\end{figure}
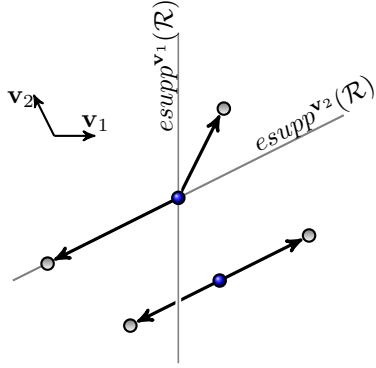

\noindent Let $\cal A$ be a set of points and ${\bf v}$ a vector in the plane. 

\begin{definition}\label{def:esup}
(i) Let $supp^{\bf v}({\cal A})_{\ge 0}$ denote the minimal half-plane which contains $\cal A$ and the positive direction of $\bf v$ and is bounded by a line perpendicular to $\bf v.$
The $\bf v$-support of $\cal A$ is the line that bounds $supp^{\bf v}({\cal A})_{\ge 0},$ 
$$supp^{\bf v}({\cal A}) = \partial supp^{\bf v}({\cal A})_{\ge 0}.$$
We define $supp^{\bf v}(\O)=\O.$ 

\noindent (ii) The ${\bf v}$-essential subnetwork of a reaction network $\cal R$ is the set of reactions in $\cal R$ whose reaction vectors are not orthogonal to $\bf v,$ i.e.
$${\cal R}_{\bf v} = {\cal R}\backslash \{P\to P'\mid (P'-P)\cdot{\bf v}=0\}$$

\noindent (iii) The ${\bf v}$-essential support of a reaction network ${\cal R}$ is 
$$esupp^{\bf v}({\cal R}) = supp^{\bf v}({\cal{SC}}({\cal R}_{\bf v})).$$
\end{definition}
\indent Note that $\cal R_{\bf v}=\O$ implies that all reaction vectors in $\cal R$ are orthogonal to $\bf v.$ In this case we have ${\cal R}_{\bf v}=esupp^{\bf v}({\cal R})= \O.$

\noindent {\bf Example.} In Figure~\ref{Fig:esupp} the dots represent six complexes and the arrows describe the four reactions among them. The two lines shown in the picture are the essential supports corresponding to the two vectors $\bf v_1$ and $\bf v_2$.

Throughout this paper we will also use the notation
$$esupp^{\bf v}({\cal R})_{<0}=\{P\in \mathbb R^2\mid (P-Q)\cdot {\bf v} < 0\ for \ all\ Q\in esupp^{\bf v}({\cal R})\} $$
\noindent and define $esupp^{\bf v}({\cal R})_{>0}$ similarly. 
\begin{definition} A reaction network $\cal R$ is called endotactic if for any nonzero vector $\bf v$ with $\cal R_{\bf v}\neq \O$ we have 
\begin{equation}\label{star}
\{P\to P' \mid P\in esupp^{{\bf v}}({\cal R}) \text{ and }P'\in esupp^{{{\bf v}}}({\cal R})_{<0}\}=\O
\end{equation}
\end{definition}
\indent In other words, and perhaps more illuminating, a network is endotactic if and only if it passes the ``parallel sweep test" for any nonzero vector $\bf v:$  sweep the lattice plane with a line $L$ orthogonal to $\bf v,$ going in the direction of $\bf v,$ and stop when $L$ encounters a source complex corresponding to a reaction which is not parallel to $L$. Now check that no reactions with source on $L$ points towards the swept region. Note that if $\cal R_{\bf v}=\O$ then $L$ never stops. In this case we still say that the network has passed the parallel sweep test for $\bf v.$

It is easy to verify that a reaction network is endotactic. Although, by definition, condition (\ref{star}) needs to be met by any vector $\bf v,$ it is easy to see that only a finite number of vectors $\bf v$ need to satisfy (\ref{star}) in order to conclude that the network is endotactic. This fact is explained in Proposition \ref{prop:endo}, which we discuss next.

Let $\bf V$ denote the set of inward normal unit vectors to the sides of $conv({\cal{SC(R)}}).$ If $conv({\cal{SC(R)}})$ is a line segment, we consider both normal unit vectors as inward.      

\begin{prop}\label{prop:endo}
A reaction network $\cal R$ is endotactic if and only if condition (\ref{star}) holds for any vector $\bf v\in \bf V\cup \{\pm\bf i, \pm\bf j\}$, where $\{\bf i, \bf j\}$ is the standard base of the cartesian plane.
\end{prop}
\begin{proof} The ``only if" implication is clear. 
For the ``if" implication, let $\bf v$ be any unit vector such that $\mathcal R_{\bf v}\neq 0.$ If $\bf v \in \bf V\cup \{\pm\bf i, \pm\bf j\}$, there is nothing to prove. Else,  if $[Q_1\ldots Q_l]$ denotes the polygon $\partial conv({\cal{SC(R)}}),$ assume that  
$esupp^{\bf v}({\cal R}) = \{Q_k\}.$ 

%
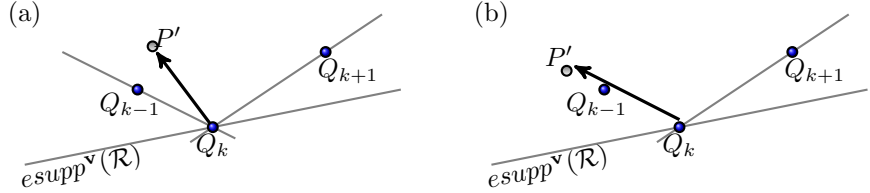
\begin{figure}[h]
\begin{center}
\begin{tabular}{cc}
\begin{tikzpicture}[thick, >=stealth', scale = .5] 

\draw[gray] (-.6,-.4) -- (4.5,3);
\draw[gray] (.6,-.3) -- (-4,2);
\draw[gray] (-5,-1) -- (5,1);
\draw[->, very thick] (0,0) -- (-1.5,2);

\shadedraw [] (-1.5,2)+(-.1,.15) circle (.13cm);
\shadedraw [shading=ball] (0,0) circle (.13cm);
\shadedraw [shading=ball] (3,2) circle (.13cm);
\shadedraw [shading=ball] (-2,1) circle (.13cm);

\node at (-5,3) {(a)};
\node at  (-1.2,2.4) {$P'$};
\node at  (0,-.5) {$Q_k$};
\node at  (-2.2,.6) {$Q_{k-1}$};
\node at  (3.6,1.5) {$Q_{k+1}$}; 
\node[rotate=12] at (-3.5,-1.1) {$esupp^{{\bf v}}(\cal R)$};
\end{tikzpicture} 
&
\hspace{.3cm}
\begin{tikzpicture}[thick, >=stealth', scale = .5] 
\draw[gray] (-.6,-.4) -- (4.5,3);
\draw[gray] (-5,-1) -- (5,1);

\shadedraw [shading=ball] (0,0) circle (.13cm);
\shadedraw [shading=ball] (3,2) circle (.13cm);
\shadedraw [shading=ball] (-2,1) circle (.13cm);
\shadedraw [] (-3,1.5) circle (.13cm);
\node at (-5,3) {(b)};
\node at  (-3.2,1.9) {$P'$};
\node at  (0,-.5) {$Q_k$};
\node at  (-2.2,.6) {$Q_{k-1}$};
\node at  (3.6,1.5) {$Q_{k+1}$}; 
\node[rotate=12] at (-3.5,-1.1) {$esupp^{{\bf v}}(\cal R)$};
\pgftransformshift{\pgfpoint{0cm}{.2cm}} 
\draw[->, very thick] (0,0) -- (-2.8,1.43);
\end{tikzpicture} 

\end{tabular}
\end{center}
\caption{Possible reactions $Q_k\to P'$ in the proof of Proposition \ref{prop:endo}. In (a), the lines $Q_{k-1}Q_{k}$ and $Q_kQ_{k+1}$ are essential supports for their inward normal vectors; the target $P'$ belongs to the interior of the angle $\widehat{Q_{k-1}Q_kQ_{k+1}}$. In (b), only $Q_kQ_{k+1}$ is essential support; $P'$ is constrained on the half line $[Q_kQ_{k-1}$. In both cases, $P'\in esupp^{\bf v}({\cal R})_{>0}.$ }\label{Fig:proofProp}
\end{figure}

First suppose  that $conv({\cal{SC(R)}})$ is not contained in a line. 
Let $\bf v_{k-1}$ and $\bf v_k$ denote inward normal vectors to $Q_{k-1}Q_{k}$ and $Q_{k}Q_{k+1}.$ The source complex $Q_k$ may belong to both or exactly one of the sets $esupp^{\bf v_k}(\cal R)$ and $esupp^{\bf v_{k+1}}(\cal R).$ 

If $Q_k\in esupp^{\bf v_k}({\cal R})\cap esupp^{\bf v_{k+1}}(\cal R),$ then (see Figure \ref{Fig:proofProp} (a)) condition (\ref{star}) applied to $\bf v_{k}$ and $\bf v_{k+1}$ shows that the target $P'$ of any reaction $Q_k\to P'$ must belong to the closed positive cone generated by the vectors $\overrightarrow{Q_{k}Q_{k-1}}$ and $\overrightarrow{Q_{k}Q_{k+1}}.$ 
This cone (except its vertex $Q_k$) is contained in $esupp^{v}({\cal R})_{>0}.$

If $Q_k\in esupp^{\bf v_k}({\cal R})\backslash esupp^{\bf v_{k+1}}(\cal R),$ all reactions $Q_k\to P'$ are along the line $Q_{k-1}Q_k.$ In fact  condition (\ref{star}) applied to $\bf v_{k+1}$ shows that they must have the direction of $\overrightarrow{Q_kQ_{k-1}}.$ Therefore (see Figure \ref{Fig:proofProp} (b))  $P'\in  esupp^{\bf v}({\cal R})_{>0}.$

Finally, if $conv({\cal{SC(R)}}) = [Q_1Q_2]$ is a line segment then (\ref{star}) applied to the two normal vectors of the line $Q_1Q_2$ shows that all target complexes lie on the line $Q_1Q_2.$ Moreover, (\ref{star}) applied to $\{\pm \bf i,\pm \bf j\}$ guarantees that in fact the all target complexes of reactions with sources $Q_1$ and $Q_2$ lie on the {\em segment} $[Q_1Q_2]$. It is now clear that (\ref{star}) holds for any vector $\bf v.$  
\end{proof}

Using Proposition \ref{prop:endo}, it is straightforward to check whether a network is endotactic or not.
A few examples are presented in Figure \ref{Fig:endoExample}.  

%
\begin{figure}
\begin{center}
\begin{tabular}{cccc}
\begin{tikzpicture}[thick, >=stealth', scale = .55] 

\node at  (-3,2) {(a)};

\draw[thin, gray] (-.5,-2.5) -- (-3.5,.5); 
\draw[thin, gray] (-3.6,-.4) -- (.6,2.4); 
\draw[thin, gray] (-.3,2.6) -- (1.8,-1.6);
\draw[thin, gray] (2,-.8) -- (-1.5,-2.2);
\draw[thin, gray] (-2.5,1) -- (.5,-2);

\draw[<->,very thick] (-1.15,-1.85) -- (-2.85,-.15); 
\draw[->,very thick] (-1,-.5) -- (2,1.5); 
\draw[->,very thick] (1.5,-1) -- (-.85,-.53);
\draw[<-,very thick] (1.4,-.8) -- (0,2);

\shadedraw [shading=ball] (-1,-2) circle (.13cm);
\shadedraw [shading=ball] (-3,0) circle (.13cm);
\shadedraw [shading=ball] (0,2) circle (.13cm);
\shadedraw [] (2,1.5)+(.15,.11) circle (.13cm);
\shadedraw [shading=ball] (1.5,-1) circle (.13cm);
\shadedraw [shading=ball] (-1,-.5) circle (.13cm);
\end{tikzpicture} 
&
\begin{tikzpicture}[thick, >=stealth', scale = .55] 

\node at  (-3,2) {(b)};

\draw[thin, gray] (-.5,-2.5) -- (-3.5,.5); 
\draw[thin, gray] (-3.6,-.4) -- (.6,2.4); 
\draw[thin, gray] (-.3,2.6) -- (1.8,-1.6);
\draw[thin, gray] (2,-.8) -- (-1.5,-2.2);
\draw[thin, gray] (-2.5,1) -- (.5,-2);
\draw[thin, gray] (-2,1.5) -- (-.25,-2);

\draw[<->,very thick] (-1.15,-1.85) -- (-2.85,-.15); 
\draw[->,very thick] (-1,-.5) -- (2,1.5);
\draw[<-, very thick] (1.325,-.965) -- (-1,-.5);
\draw[<->,very thick] (1.4,-.8) -- (.1,1.8); 
\shadedraw [shading=ball] (-1,-2) circle (.13cm);
\shadedraw [shading=ball] (-3,0) circle (.13cm);
\shadedraw [shading=ball] (0,2) circle (.13cm);
\shadedraw [] (2,1.5)+(.15,.11) circle (.13cm);
\shadedraw [shading=ball] (1.5,-1) circle (.13cm);
\shadedraw [shading=ball] (-1,-.5) circle (.13cm);
\end{tikzpicture} 
&
\begin{tikzpicture}[thick, >=stealth', scale = .55] 

\node at  (-1.6,2) {(c)};

\draw[thin, gray] (-2,-2) -- (2.5,2.5); 
\draw[thin, gray] (-1.6,-2.4) -- (-1.6,1.7); 
\draw[thin, gray] (-2.4,-1.6) -- (1.7,-1.6); 
\draw[thin, gray] (1,-1.6) -- (1,2.8); 
\draw[thin, gray] (-1.6,1) -- (2.8,1); 

\draw[thin, gray] (-2,-2) -- (2.5,2.5); 
\pgftransformshift{\pgfpoint{0cm}{.3cm}} 
\draw[->,very thick] (-.05,-.05)  -- (1.85,1.85); 
\draw[->,very thick] (-1.7,-1.7)  -- (-.15,-.15); 
\pgftransformshift{\pgfpoint{0cm}{-.3cm}} 
\shadedraw [shading=ball] (-1.6,-1.6) circle (.13cm);
\shadedraw [] (2,2) circle (.13cm);
\shadedraw [shading=ball] (0,0) circle (.13cm);
\shadedraw [shading=ball] (1,1) circle (.13cm);
\pgftransformshift{\pgfpoint{0cm}{-.3cm}} 
\draw[->,very thick] (1.1,1.1)  -- (-1.4,-1.4); 
\end{tikzpicture} 
&
\begin{tikzpicture}[thick, >=stealth', scale = .55] 

\node at  (-1.6,2) {(d)};

\draw[thin, gray] (-2,-2) -- (2.5,2.5); 
\draw[thin, gray] (-1.6,-2.4) -- (-1.6,1.7); 
\draw[thin, gray] (-2.4,-1.6) -- (1.7,-1.6); 
\draw[thin, gray] (1,-1.6) -- (1,2.8); 
\draw[thin, gray] (-1.6,1) -- (2.8,1); 

\draw[thin, gray] (-2,-2) -- (2.5,2.5); 
\pgftransformshift{\pgfpoint{0cm}{.3cm}} 
\draw[->, very thick] (-.05,-.05)  -- (0.85,0.85); 
\draw[->, very thick] (-1.7,-1.7)  -- (-.15,-.15); 

\pgftransformshift{\pgfpoint{0cm}{-.3cm}} 

\draw[->, very thick] (1,1)  -- (1.85,1.85); 
\shadedraw [shading=ball] (-1.6,-1.6) circle (.13cm);
\shadedraw [] (2,2) circle (.13cm);
\shadedraw [shading=ball] (0,0) circle (.13cm);
\shadedraw [shading=ball] (1,1) circle (.13cm);
\pgftransformshift{\pgfpoint{0cm}{-.3cm}} 
\draw[->, very thick] (1.1,1.1)  -- (-1.4,-1.4); 
\end{tikzpicture} 

\end{tabular}
\end{center}
\caption{Examples of endotactic networks -- (a) and (c), and non-endotactic networks -- (b) and (d). If the convex hull of the sources is a non-degenerate polygon ((a) and (b)), the parallel sweep test is done with vectors of $\bf V.$ If the sources are confined to a line ((c) and (d)), the parallel sweep test for $\{\pm\bf i,\pm\bf j\}$ needs to be performed also; note that using just the vectors in $\bf V$ would have concluded that the network (d) is endotactic.}\label{Fig:endoExample}
\end{figure}
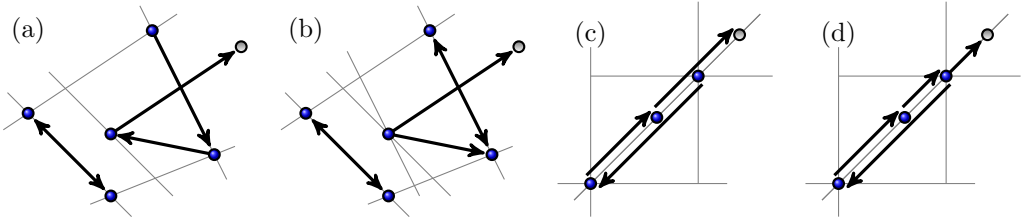

We have mentioned that the class of endotactic networks includes the class of weakly reversible networks. Indeed, this is easy to see:

\begin{lemma}\label{domReact}
Any weakly reversible reaction network is endotactic.
\end{lemma}
\begin{proof}
In a weakly reversible network all complexes are sources, ${\cal{SC(R)}}=\cal S.$ The definition of the essential support implies immediately that for every vector $\bf v$ all reactions originating on $esupp^{\bf v}(\cal R)$ point inside $esupp^{\bf v}({\cal R})_{>0}.$
\end{proof}

Because of its relation to complex-balancing, weak reversibility is a widespread assumption in much of the work done in Chemical Reaction Network Theory (\cite{cinci}). On the other hand, the results of the present paper naturally lend themselves to the much larger class of endotactic networks. In the subsequent sections we will prove persistence and permanence of endotactic $\kappa$-variable mass-action systems with two species. We conjecture that the same results are true for any number of species: 

\smallskip
\noindent {\bf Extended Persistence Conjecture.} {\em Any endotactic $\kappa$-variable mass-action system is persistent.}\\
\noindent {\bf Extended Permanence Conjecture.} {\em Any endotactic $\kappa$-variable mass-action\footnote{Same for power-law.} system is permanent.}
\smallskip

Here ``extended" refers to two different generalizations: the assumption of an endotactic, and not necessarily weakly reversible network, and for the use of the $\kappa$-variable mass-action kinetics, and not the usual mass-action kinetics, where $\kappa$ is a fixed vector of parameters. 

\section{Construction of a forward-invariant polygon}\label{sec:geom} 
As in the proof of persistence of the reversible example (\ref{ex:Illustrative}), our approach to proving Theorem \ref{thm:main} is based on defining a convex polygon $\cal P$ whose convex hull contains $T(c(0))$. Here we explain how this construction is carried out for a not necessarily reversible network. 

In example (\ref{ex:Illustrative}) the aggregate effect of any two reverse reactions $P\rightleftharpoons P'$ is to drive the trajectory towards the $\Gamma$ region $\delta \le c^{P'-P}\le 1/\delta,$ therefore towards the interior of $\cal P.$ In other words, the source of the ``bad" reaction that might force the trajectory to cross to the exterior of $\cal P$ is smaller (up to  some constant) than the source of the ``good" reaction that keeps the trajectory in the interior of $\cal P.$

The proof of persistence for example  (\ref{ex:Illustrative}) contains all the ingredients for the proof of the persistence for any two-species reversible $\kappa$-variable mass-action system. In the more general setup of a weakly reversible network, we need to adjust the setup of the invariant polygon and parts of our discussion. 

In the reversible case we saw that if $c(t_0)\in \cal P,$ every pair of reactions $P\rightleftharpoons P'$ contains a ``good" reaction that dictates the overall effect of the reversible reaction. In the weakly reversible case we show that there is a ``good" reaction that dictates the direction of $\dot c(t_0),$ i.e. {\em the aggregate effect of the whole network}, up to a small enough perturbation. This amounts to showing that at time $t_0$ with $c(t_0)\in \cal P$ there is a complex that is source to a ``good" reaction and whose corresponding monomial is larger than the monomials corresponding to all other source complexes (up to  some constant). This way we are led to comparing monomials $c^P$ and $c^{P'}$ for {\em all pairwise combinations of source complexes} $P$ and $P'.$

Let us fix a trajectory $T(c(0))$ of (\ref{def:massAct}) and denote $c(t)=(x(t),y(t))$ for $t\geq 0.$ 
Let $(1,r_1)\ldots, (1,r_e)$ and $(1,s_1)\ldots, (1,s_f)$ be normal vectors of edges with finite and nonzero slope in the complete graph with vertices ${\cal {SC}}(\cal R):$
$$\{r_1,\ldots, r_e, s_1,\ldots, s_f\} = \left\{\frac{m_1-m_2}{n_2-n_1}\mid (m_1,n_1),(m_2,n_2)\in {\cal {SC}}({\cal R}), m_1\neq m_2, n_1\neq n_2 \right\} $$
\noindent where $s_1<\ldots<s_f < 0 < r_1<\ldots<r_e.$

Let $0<\delta<1$ be fixed. The pairwise comparison of source complexes $P=(m,n)$ and $(m',n')$ up to  a constant leads to $c^P-\delta c^{P'}=0$ and $c^P-(1/\delta) c^{P'}=0$, or $y=\delta^{1/(n-n')} x^{\sigma}$ and $y=(1/\delta)^{1/(n-n')} x^{\sigma},$  where $\sigma\in \{r_1,\ldots, r_e, s_1,\ldots, s_f\}.$ Let
$$\delta' = \min(\delta^{1/(n'-n)}\mid (m,n),(m',n')\in {\cal{SC(R)}}, n\neq n')$$ 
\noindent and consider the curves 
$$y=\delta' x^{r_i}, y=(1/\delta') x^{r_i}  \text{ and } y=\delta' x^{s_j}, y=(1/\delta') x^{s_j} \text{ for all } i\in \{1,\ldots, e\} \text{ and } j\in \{1,\ldots, f\},$$
\noindent depicted using dots in Figure \ref{Fig:bigPicture}.
We further let $r_{i+.5}$ be positive reals and $s_{j+.5}$ be negative reals for all $i\in \{0,\ldots, e\}$  and  $j\in \{0,\ldots, f\}$ such that 
$$r_{i-.5}<r_{i}<r_{i+.5}\text{ and }s_{i-.5}<s_{i}<s_{i+.5}.$$
\noindent The solid curves in Figure \ref{Fig:bigPicture} are given by 
$$y=x^{r_{i+.5}} \text{ and } y=x^{s_{j+.5}},  i\in \{0,\ldots, e\} \text{ and } j\in \{0,\ldots, f\}.$$

\begin{remark}\label{rem:frac}
In our construction, the vertices   of $\cal P$ will be confined to these ``fractional indices" curves. Although this constraint is not necessary and we may allow more freedom in choosing the vertices of $\cal P$ as in the reversible example  (\ref{ex:Illustrative}),  our construction conveniently makes $\cal P$ depend continuously on a single parameter. 
\end{remark}

Let $\xi<1$ be positive and small enough and $M>1$ be positive and large enough such that (see Figure \ref{Fig:bigPicture}):
\begin{myindentpar}{1cm}
(P1) $c(0)\in (\xi,M)^2$;\\
(P2) all intersection of the curves in Figure \ref{Fig:bigPicture} lie in the square $(\xi,M)^2$;\\
(P3) $(0,\xi)^2$ and $(M,\infty)^2$ lie below, respectively above all negative-exponent ($s_j$ or $s_{j+.5}$) curves. Likewise, $(0,\xi)\times (M,\infty)$ and  $(M,\infty)\times (0,\xi)$ lie above, respectively below all curves of positive exponent ($r_i$ or $r_{i+.5}$);\\
(P4) All negative exponent curves intersect the line segments $(0,\xi)\times \{M\}$ and $\{M\}\times (0,\xi)$;\\
(P5)
\vspace{-.8cm}
\begin{align}
\nonumber
\xi<\min(\delta^{1/(n'-n)}\mid (m,n),(m',n')\in {\cal{SC(R)}}, n\neq n'),\\\nonumber
\xi<\min(\delta^{1/(m'-m)}\mid (m,n),(m',n')\in {\cal{SC(R)}}, m\neq m'),\\\nonumber
M > \max((1/\delta)^{1/(n'-n)}\mid (m,n),(m',n')\in {\cal{SC(R)}}, n\neq n'),\\\nonumber
M > \max((1/\delta)^{1/(m'-m)}\mid (m,n),(m',n')\in {\cal{SC(R)}}, m\neq m').
\end{align}
\end{myindentpar}
%
\begin{figure}
\begin{center}
\begin{tikzpicture}[>=stealth, scale =.58]
\fill[color = blue!25!white]  (0,0) rectangle +(4,4); 
\shade[top color=white, bottom color = blue!25!white] (0,12) rectangle +(4,5); 
\shade[right color=white, left color = blue!25!white] (12,0) rectangle +(5,4); 
\begin{scope}[draw = white]
	\draw[clip] (12,12) rectangle + (5,5);
	\shadedraw[inner color = blue!25!white,outer color=white,draw=white] (7,7) rectangle +(10,10); 
\end{scope}
\draw[->](0,0)--(17,0) node [below] {$x$};
\draw[->](0,0)--(0,17) node [left] {$y$};;
\draw[very thin, gray](4,0)--(4,17); 
\draw[very thin, gray](0,4)--(17,4);
\draw[very thin, gray](0,12)--(17,12);
\draw[very thin, gray](12,0)--(12,17);
\node at  (4,-.5) {$\xi$};
\node at  (12,-.5) {$M$};
\node at  (-.5,4) {$\xi$};
\node at  (-.5,12) {$M$};
\begin{scope}
\clip (0,0) rectangle + (17,17);
\draw[thick,dotted] plot[domain = 0:13.7] (\x,.95*\x^2/11);
\draw[thick, dotted] plot[domain = 0:13.6] (\x,1.01*\x^2/11);
\draw[thick,dotted] plot[smooth, domain = 0.001:16.9] (.092*\x^2,\x);
\draw[thick,dotted] plot[smooth, domain = 0.001:16.9] (.084*\x^2,\x);
\draw[blue,thick](.5,3.5)--(1,2);
\node[rotate = -45] at  (1.5,1.5) {$\ldots$};
\draw[blue,thick] (3.5,.5)--(2,1);
\fill[blue] (.5,3.5) circle (.1);
\fill[blue] (1,2) circle (.1);
\fill[blue] (2,1) circle (.1);
\fill[blue] (3.5,.5) circle (.1);
\node at  (.4,3.9) {$A_1$};
\node at  (1.4,2) {$A_2$};
\node at  (2,1.4) {$A_{e}$};
\node at  (3.6,.23) {$A_{e+1}$};

\draw[blue,thick] (3.5,.5)--(13,.5);
\draw[blue,thick] (13,.5)--(14.3,1.2);
\node[rotate = 68] at  (14.7,2) {$\ldots$};
\draw[blue,thick] (14.9,2.7)--(15.2,3.7);
\fill[blue] (13,.5) circle (.1);
\fill[blue] (14.3,1.2) circle (.1);
\fill[blue] (14.9,2.7) circle (.1);
\fill[blue] (15.2,3.7) circle (.1);
\node at  (13,.23) {$B_1$};
\node at  (14.4,.8) {$B_2$};
\node at  (15.3,2.4) {$B_{f}$};
\node at  (15.9,4.04) {$B_{f+1}$};
\draw[blue,thick] (15.2,3.7)--(15.2,12.3);
\draw[blue,thick] (15.2,12.3)--(14.7,13.7);
\node[rotate = -55] at  (14.2,14.6) {$\ldots$};
\draw[blue,thick] (13.6,15.4)--(12.6,16);
\fill[blue] (15.2,12.3) circle (.1);
\fill[blue] (14.7,13.7) circle (.1);
\fill[blue] (13.6,15.4) circle (.1);
\fill[blue] (12.6,16) circle (.1);
\node at  (15.6,12) {$C_1$};
\node at  (15.2,13.6) {$C_2$};
\node at  (14.1,15.4) {$C_{e}$};
\node at  (12.3,16.4) {$C_{e+1}$};

\draw[blue,thick] (12.6,16)--(3.8,16);
\draw[blue,thick] (3.8,16)--(2.7,15.8);
\node[rotate = 29] at  (1.8,15.4) {$\ldots$};
\draw[blue,thick] (.5,14)--(1.1,15);
\fill[blue] (3.8,16) circle (.1);
\fill[blue] (2.7,15.8) circle (.1);
\fill[blue] (.5,14) circle (.1);
\fill[blue] (1.1,15) circle (.1);
\node at  (4.2,16.4) {$D_1$};
\node at  (2.4,16.2) {$D_2$};
\node at  (1.3,14) {$D_{f+1}$};
\node at  (1.6,14.8) {$D_{f}$};

\draw[blue,thick] (.5,3.5)--(.5,14);

\draw (0,0) .. controls +(90:.5cm) and +(240:1cm)..(1,2);
\draw (1,2) .. controls + (60:5cm) and +(210:3cm).. (14.7,13.7);
\draw(14.7,13.7).. controls +(30:.5cm) ..  (16,14.5) ;

\draw (0,0) .. controls +(90:1cm) and +(240:1cm)..(.5,3.5);
\draw (.5,3.5) .. controls + (60:5cm) and +(190:4cm).. (15.2,12.3);
\draw(15.2,12.3).. controls +(10:.5cm) ..  (16,12.4) ;

\draw (0,0) .. controls +(0:1cm) and +(210:1cm)..(3.5,.5);
\draw (3.5,.5) .. controls + (30:5cm) and +(260:4cm).. (12.6,16);
\draw(12.6,16).. controls +(80:1cm) ..  (12.75,17) ;

\draw (0,0) .. controls + (0:.5cm) and +(210:1cm)..(2,1);
\draw(2,1) .. controls + (30:5cm) and +(250:4cm).. (13.6,15.4);
\draw(13.6,15.4).. controls + (70:1cm) ..  (14.1,17) ;

\draw(3.8, 16) --(3.8,17);
\draw (3.8,16) .. controls + (270:2cm) and + (100:1cm) .. (4,12);
\draw (7,4.5) .. controls + (130:3cm) and +(280:1cm).. (4,12);
\draw (7,4.5) .. controls + (310:2cm) and +(160:2cm).. (13,.5);
\draw (16,.3) .. controls + (180:1cm) and +(340:.5cm).. (13,.5);

\draw[thick,dotted] (3.5,16.5) .. controls + (272:9cm) and + (173:9cm) .. (15.5,.8);
\draw[thick,dotted] (3,16.5) .. controls + (273:8cm) and + (170:9cm) .. (15,.5);

\draw(15.2, 3.7) --(16,3.7);
\draw (15.2,3.7) .. controls + (180:1cm) and + (350:1cm) .. (12,4);
\draw (.5,14) .. controls + (280:7cm) and + (170:3cm) .. (12,4);
\draw (.5,14) .. controls + (100:1cm) .. (.3,17);

\draw (1.1,15) .. controls + (290:7cm) and + (170:9cm) .. (14.9,2.7);
\draw (1.1,15) .. controls + (110:.5cm) .. (.7,17);
\draw (14.9,2.7) .. controls + (350:.1cm) .. (16,2.6);

\draw[thick,dotted] (.6,16) .. controls + (280:7cm) and + (175:9cm) .. (15.5,3.4);
\draw[thick,dotted] (.5,15) .. controls + (280:5cm) and + (175:9cm) .. (15.5,2.9);

\draw (2.7,15.8) .. controls + (280:7cm) and + (170:9cm) ..  (14.3,1.2);
\draw (2.7,15.8) .. controls + (100:.5cm) .. (2.5,17);
\draw (14.3,1.2) .. controls + (350:.1cm) .. (16,1.1);
\end{scope}
\end{tikzpicture} 
\end{center}
\caption{Setup of the invariant polygon ${\cal P} = A_1\ldots A_{e+1}B_1\ldots B_{f+1}C_1\ldots C_{e+1}D_1\ldots D_{f+1}$}\label{Fig:bigPicture}
\end{figure}

%
\begin{figure}
\begin{center}
\begin{tikzpicture}[>=stealth, scale =.38]

\node at (16.2,-.7) {$(0,\xi)$};
\node at (0,16.6) {$(\xi,0)$};
\draw (0,0)--(16,0);
\draw (0,0)--(0,16);
\draw [gray](16,0)--(16,16);
\draw [gray](0,16)--(16,16);
\draw[dotted, thick] plot[domain = 0:16] (\x,\x^2/130);
\draw[dotted, thick] plot[domain = 0:16] (\x,\x^2/80);
\draw[dotted, thick] plot[domain = 0:16] (\x^2/60,\x);
\draw[dotted, thick] plot[domain = 0:16] (\x^2/90,\x);
\draw[dotted, thick] plot[domain = 0.0001:16] (\x^3/320,\x^2/16);
\node[rotate = 38, fill=white, text = black] at (10.5,14) {$y=\frac{1}{\delta'}x^{r_2}$};
\draw[dotted, thick] plot[domain = 0.0001:16] (\x^3/256,\x^2/16);
\node[rotate = 37, fill=white, text = black] at (12.9,14) {$y=\delta' x^{r_2}$};

\draw plot[domain = 0:.842] (\x^2,19*\x);
\draw plot[domain = 0:2.832] (\x^2,5.65*\x);
\draw plot[domain = 0:1.3195] (\x^{10},.945*\x^{9});
\draw plot[domain = 0:2.5] (6.4*\x,\x^2);
\draw plot[domain = 0:.842] (19*\x,\x^2);

\node[rotate = 46, fill=white, text = black] at (6.3,14) {$y=x^{r_{1.5}}$};
\node[rotate = 35, fill=white, text = black] at (12,9) {$y=x^{r_{2.5}}$};

\draw [blue, thick](.5,13.25)--(.5,16);
\draw [blue, thick](.5,13.25)--(2,8);
\draw [blue, thick](2,8)--(5,4);
\draw [blue, thick, dashed](5,4)--(9,2);
\draw [blue,thick](9,2)--(13.5,.5);
\draw [blue,thick](13.5,.5)--(16,.5);

\fill[blue] (.5,13.25) circle (.15);
\fill[blue] (2,8) circle (.15);
\fill[blue] (5,4) circle (.15);
\fill[blue] (9,2) circle (.15);
\fill[blue] (13.5,.5) circle (.15);
\node at (1.2,13.25) {$A_1$};
\node at (1.3,8) {$A_2$};
\node at (4.3,4) {$A_3$};
\node at (8.6,1.6) {$A_e$};
\node at (13.9,1) {$A_{e+1}$};
\node[rotate = 60, fill=white, text=blue] at (7,3) {\normalsize$\approx$};
\draw [gray](0,16)--(16,16);

\draw (0,0)--(16,0);
\draw (0,0)--(0,16);
\draw [gray](16,0)--(16,16);
\draw [gray](0,16)--(16,16);

\end{tikzpicture} 
\end{center}
\caption{Detail of vertices $A_1,\ldots, A_{e+1}$}\label{Fig:colt}
\end{figure}
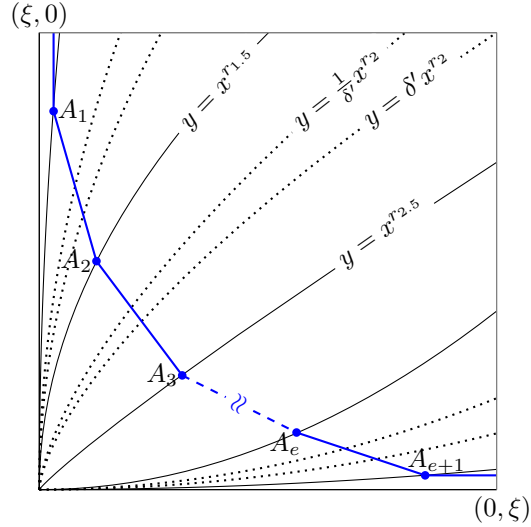

The goal is to obtain ${\cal P}=A_1\ldots A_{e+1}B_1\ldots B_{f+1}C_1\ldots C_{e+1}D_1\ldots D_{f+1}$  as in Figure \ref{Fig:bigPicture} with the property that 

\noindent (P*) {\em the dotted curves 
$y = \delta' x^{\sigma}$ and $y = (1/\delta') x^{\sigma}$ 
intersect the two sides of $\cal P$ that are orthogonal to the vector $(1, \sigma)$ and only these two sides.}

Conditions (P1--P5) allow for the construction of such a polygon, and moreover, assure that $\cal P$ is invariant, as we will see in the proof of Theorem \ref{thm:main}. Due to (P4), the $B_i$ and $D_i$ vertices are easily constructed in the correct shaded regions. Condition (P2) guarantees that on the shaded areas the pair of dotted curves $y=\delta' x^{\sigma},\ y=(1/\delta') x^{\sigma},$ $\sigma\in\{r_1,\ldots, r_e, s_1,\ldots, s_f\}$ lie between the corresponding solid curves, an important fact  that we will take advantage of later. Conditions (P3) and (P5) will also come in handy, as it will make some pairwise comparisons of source monomials straightforward. Finally, condition (P1) make the initial condition lie in $conv(\cal P).$

We construct ${\cal P}=A_1\ldots A_{e+1}B_1\ldots B_{f+1}C_1\ldots C_{e+1}D_1\ldots D_{f+1}$ one vertex at the time starting with $A_1 = (\alpha,\alpha^{s_{.5}})$ and going counterclockwise. All vertices of $\cal P$ lie on a fractional index curve and satisfy property (P*) above. In other words the sides $A_iA_{i+1},$ $C_iC_{i+1}$ are orthogonal to $(1,r_i)$ and the sides  $B_jB_{j+1}$ and $D_jD_{j+1}$ are orthogonal to $(1,s_j).$ We also require that $A_{e+1}B_1,$  $C_{e+1}D_1$ are horizontal and that $B_{f+1}C_1$ is vertical. The polygon $\cal P$ constructed this way is unique.

The vertices of $\cal P$ can indeed be chosen to lie in the four shaded areas of Figure \ref{Fig:bigPicture}. For instance,  $\lim_{\alpha\to 0} A_i = (0,0)$; therefore $A_i\in (0,\xi)^2$ for all $i\in \{1,\ldots, e+1\}$  for small $\alpha,$ whereas (P4) guarantees that $B_j\in (M,\infty)\times (0,\xi)$ for all $j.$ Finally, note that the side $D_{f+1}A_1$ of $\cal P$ need not be vertical, but in the limit $\alpha\to 0,$ $D_{f+1}A_1$ is vertical. In particular, ${\cal P}(\alpha)$ is convex for small $\alpha.$

Let $\alpha_0$ be such that ${\cal P}(\alpha_0)$ is a convex polygon whose construction is as explained above.  Then $\alpha\in(0,\alpha_0]$ defines  a one-parameter family of convex polygons;  the generic ${\cal P}(\alpha)$  is outlined in Figure \ref{Fig:bigPicture}. The vertices $A_1,\ldots, A_e$ are illustrated in more detail in Figure \ref{Fig:colt}.

Finally, let us remark that $P(\alpha)$ varies continuously with $\alpha$ and $conv({\cal P}(\alpha))$ is decreasing in $\alpha\in(0,\alpha_0)$ and is a cover of $\mathbb R_{>0}^2:$
\begin{equation}\label{covers}
conv({\cal P}(\alpha'))\supset conv({\cal P}(\alpha'')) \text{ if }\alpha'<\alpha''\text{ and }\bigcup _{\alpha=0}^{\alpha_0}conv({\cal P}(\alpha))=\mathbb R_{>0}^2.
\end{equation}

\section{Endotactic two-species $\kappa$-variable mass-action systems: persistence and permanence}\label{section:proof}
\subsection{Persistence} Let $(\cal R, \cal S, \cal C, \kappa)$ denote a two-species, endotactic $\kappa$-variable mass action system with $\kappa_{P\to P'}(t)\in (\eta,1/\eta)$ for all $P\to P'\in \cal R.$ Recall that $c(t)=(x(t),y(t))$ denotes a solution of (\ref{varMassAct}) with initial condition $c(0).$

The following lemma is a key ingredient in the proof of Theorem \ref{thm:main}. It states that if for some $t_0$, the source monomial of the reaction $P_0\to P'_0$ is larger than all other source monomials (up to  a constant), then $\dot c(t_0)$ does not deviate too far from $P'_0-P_0.$
\begin{lemma}\label{staysInside}
Let $P_0\to P'_0\in {\cal R}$ and let $\bf{v}$ be a vector such that $(P'_0-P_0)\cdot {\bf v}>0.$ Also let  ${\cal U}\subseteq {\cal{SC}}({\cal R})\backslash \{P_0\}.$ 
There exists a constant $\delta$ such that if for some $t_0\ge 0$ we have $c(t_0)\in \mathbb R_{>0}^2$ and
$${c(t_0)}^P<\delta {c(t_0)}^{P_0} \text{ for all }P\in \cal U$$
\noindent then
$$\left(\kappa_{P_0\to P'_0}(t_0)(P'_0-P_0){c(t_0)}^{P_0}+\sum_{P\to P', P\in {\cal U} }\kappa_{P\to P'}(t_0) (P'-P){c(t_0)}^P\right)\cdot{\bf{v}} > 0.$$
\end{lemma}
\begin{proof} We take 
$$\delta = \frac{\eta^2 (P'_0-P_0)\cdot \bf{v}}{\vectornorm{{\bf v}}\displaystyle\sum_{P\to P'\in{\cal R}}\vectornorm{P'-P}}$$
\noindent and we have 
\begin{eqnarray*} \label{case2}
&&\left(\kappa_{P_0\to P'_0}(t_0)(P'_0-P_0){c(t_0)}^{P_0}+\sum_{P\to P', P\in {\cal U}}\kappa_{P\to P'}(t_0) (P'-P){c(t_0)}^P\right)\cdot{\bf v}\ge \\\nonumber
&&\ge \kappa_{P_0\to P'_0}(t_0)c(t_0)^{P_0}(P'_0-P_0)\cdot {\bf v} - \sum_{P\to P', P\in {\cal U}} \Bigg(\kappa_{P\to P'}(t_0)\vectornorm{{\bf v}}\vectornorm{P'-P} \delta c(t_0)^{P_0} \Bigg )\ge \\
&&\geq\left( \eta (P'_0-P_0)\cdot {\bf{v}} - \delta(\vectornorm{{\bf v}}/\eta)\sum_{P\to P'\in {\cal R}}\vectornorm{P'-P}\right){c(t_0)}^{P_0} = 0\\\nonumber
\end{eqnarray*} 
\end{proof}

We may now state and prove our main result.
\begin{theorem}\label{thm:main}
Any two-species, endotactic $\kappa$-variable mass-action system is persistent and has bounded trajectories. 
\end{theorem}
\begin{proof} Recall that $\bf V$ denotes the set of inward normal vectors to the sides of $conv({\cal{SC(R)}}).$ Let $\bf n\in\bf V\cup\{\pm \bf i, \pm\bf j\},$ and let $P_{\bf n}\to P'_{\bf n}$ denote a fixed reaction such as  
$P_{\bf n}\in esupp^{\bf n}(\cal R)$ and $P'_{\bf n}\in esupp^{\bf n}({\cal R})_{>0}.$ Since $\cal R$ is endotactic, such a reaction exists for all $\bf n\in \bf V\cup\{\pm \bf i, \pm\bf j\}.$

Note that including $\{\pm\bf i, \pm\bf j\}$ here guarantees that at least one vector  $\bf n\in \bf V \cup\{\pm\bf i, \pm\bf j\}$ satisfies $esupp^{\bf n}({\cal R})\neq\O.$  For every such vector $\bf n$ we let $\delta_{\bf n}$ denote the constant from Lemma \ref{staysInside} that corresponds to reaction $P_{\bf n}\to P'_{\bf n},$ $\bf v = \bf n$ and ${\cal U}={\cal{SC(R})}\cap esupp^{\bf n}({\cal R})_{>0}$ and 
\begin{equation}\label{delta} 
\delta = \min_{\bf n} \delta_{\bf n}.
\end{equation}
\noindent Then $\delta<1$ and the construction from section \ref{sec:geom} yields the convex polygon 
$${\cal P}(\alpha_0)= A_1\ldots A_{e+1}B_1\ldots B_{f+1}C_1\ldots C_{e+1}D_1\ldots D_{f+1} = \cal P\text{ for simplicity of notation}.$$ 
We view $\cal P$ as the union of its sides, so that ${\cal P} = \partial (conv(\cal P)).$ By construction, $c(0)\in (\xi,M)^2\subset conv(\cal P).$ We show that $conv(\cal P)$ is forward invariant for the dynamics (\ref{def:massAct}), i.e. $c(t)\in conv(\cal P)$ for all $t\geq 0.$ More precisely, we prove that  $c(t)$ cannot cross to the exterior of $\cal P$ by showing that if $c(t_0)\in {\cal P}$ for some $t_0\geq 0,$ then
\begin{equation}\label{mainIneq}
\dot c(t_0)\cdot {\bf n} \geq 0 \text{  whenever } c(t_0)\in {\cal P} 
\end{equation}
where $-{\bf n}$ is a generator of the normal cone $N_{\cal P}(c(t_0))$ of $\cal P$ at $c(t_0).$ Recall that 
$$N_{\cal P}(c(t_0))=\{{\bf v}\in \mathbb R^2\mid {\bf v}\cdot ({\bf x}-c(t_0)))\le 0 \text{ for all } {\bf x}\in \mathrm{conv}(\cal P)\}.$$  
Note that $n$ is orthogonal to one of the sides of $\cal P$ and points inside $\cal P.$ 

Inequality (\ref{mainIneq}) is rewritten as  
\begin{equation}\label{newMainIneq}
\left(\sum_{P\to P'\in {\cal R}_{\bf n}}\kappa_{P\to P'}(t_0) (P'-P)c(t_0)^P\right)\cdot{\bf n}\geq 0
\end{equation}

\noindent and  is satisfied if ${\cal R}_{\bf n}=\O.$ Otherwise, let $P_{\bf n}\to P'_{\bf n}\in{\cal R}_{\bf n}$ with $P_{\bf n}\in esupp^{\bf n}(\cal R)$. We rewrite the left hand side of (\ref{newMainIneq}) by separating the reactions with source on the  ${\bf n}$~-~essential support of $\cal R$ and further isolating the effect of the reaction $P_{\bf n}\to P'_{\bf n}:$
\begin{eqnarray}\label{eq:cdot}\nonumber
&&\left(\sum_{P\to P'\in {\cal R}_{\bf n}}\kappa_{P\to P'}(t_0) (P'-P)c(t_0)^P\right)\cdot {\bf n}=\\\nonumber
&&=\Bigg(\sum_{\substack{{P\to P'\neq P_{\bf n}\to P'_{\bf n}}\\{P\in esupp^{{\bf n}}({\cal R})}}}\kappa_{P\to P'}(t_0) (P'-P)c(t_0)^{P}+\\\nonumber
&&+\kappa_{P_{\bf n}\to P'_{\bf n}}(t_0) (P'_{\bf n}-P_{\bf n})c(t_0)^{P_{\bf n}}
+\sum_{P\to P'\in {\cal R}_{\bf n}, P\notin esupp^{{\bf n}}({\cal R})}\kappa_{P\to P'}(t_0) (P'-P)c(t_0)^P\Bigg)\cdot{\bf n}.\\\nonumber
\end{eqnarray}

\noindent Since all source complexes of ${\cal R}_{\bf n}$ lie in $esupp^{{\bf n}}({\cal R}_{\bf n})_{\ge 0}$, the reaction vector $P'-P$ with source $P\in esupp^{{\bf n}}(\cal R)$ satisfies $(P'-P) \cdot  {\bf n}\geq 0.$ Note that for the reaction $P_{\bf n}\to P'_{\bf n}$ the inequality is strict,
$$(P'_{\bf n}-P_{\bf n})\cdot {\bf n} > 0.$$

It is therefore enough to show that
\begin{equation}\label{eq:reduced}
\left(\kappa_{P_{\bf n}\to P'_{\bf n}}(t_0) (P'_{\bf n}-P_{\bf n})c(t_0)^{P_{\bf n}}+\sum_{\substack{{P\to P'\in{\cal R}_{\bf n}}\\ {P\notin esupp^{{\bf n}}({\cal R})}}}\kappa_{P\to P'}(t_0) (P'-P)c(t_0)^P\right)\cdot {\bf n} \geq 0
\end{equation}

\noindent in order to verify (\ref{newMainIneq}). In turn, (\ref{eq:reduced}) will follow from Lemma \ref{staysInside} with ${\cal U}={\cal{SC(R)}}\cap esupp^{{\bf n}}({\cal R})_{>0}$ and the fact that 
\begin{equation}\label{ineq}
\delta c(t_0)^{P_{\bf n}} > c(t_0)^P
\end{equation}
\noindent for all $P\in \cal U,$ inequality whose verification will complete the proof of the theorem. 

The proof of inequality (\ref{ineq}) depends on the position of $c(t_0)$ on $\cal P.$ More precisely, $c(t_0)$ might belong to one of the four polygonal lines $[A_1\ldots A_{e+1}],$ $[B_1\ldots B_{f+1}]$, $[C_1\ldots C_{e+1}]$ and $[D_1\ldots D_{f+1}]$ or to one of the four remaining sides of $\cal P.$ We discuss the first case,  $c(t_0)\in [A_1\ldots A_{e+1}].$ Suppose $c(t_0)=(x(t_0),y(t_0))$ lies on the side $[A_iA_{i+1}].$ Then (see Figure \ref{Fig:colt}) we have 
$$(1/\delta')x(t_0)^{r_{i+1}} < y(t_0) < \delta' x(t_0)^{r_{i-1}},$$
where only one side of the inequality holds if $i=1$ or $i=e.$

%
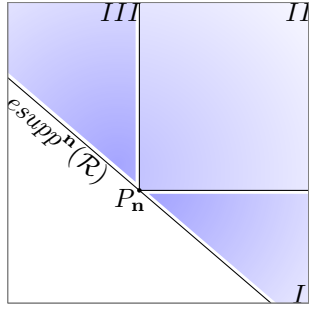
\begin{figure}
\begin{center}
\begin{tikzpicture}[>=stealth, scale =.25]

\begin{scope}
\clip (0,0) rectangle+(16,16);

\begin{scope}[shift = {(.5,-.18)}] 
\clip (7,6)--(16,6)--(16,0)--(14,0);
\shadedraw[inner color = blue!40!white,outer color=white,draw=white] (-9,-8) rectangle +(30,30);
\end{scope}

\begin{scope}[shift = {(-.2,.45)}] 
\clip (7,6)--(0,12)--(0,16)--(7,16);
\shadedraw[inner color = blue!35!white,outer color=white,draw=white] (-9,-8) rectangle +(30,30);
\end{scope}

\begin{scope}
\clip (7,6) rectangle + (10,10);
\shadedraw[inner color = blue!25!white,outer color=white,draw=white] (-9,-8) rectangle +(30,30);
\end{scope}

\node[rotate = -40] at (2.6,8.7) {$esupp^{\bf n}(\cal R)$};
\node at (6.5,5.5) {$P_{\bf n}$};
\fill[black] (7,6) circle (.12);
\node at (15.5,.5) {$I$};
\node at (15.5,15.5) {$II$};
\node at (6,15.5) {$III$};

\draw(0,12)--(14,0);
\draw(7,16)--(7,6)--(16,6);
\draw[gray](0,0)--(0,16)--(16,16)--(16,0)--(0,0); 
\draw[gray](0,0)--(0,16)--(16,16)--(16,0)--(0,0); 

\end{scope}

\end{tikzpicture} 
\end{center}
\caption{Positions of a source monomial $P$ relative to $P_{\bf n}$}\label{Fig:relPos}
\end{figure}

Let $P=(m,n)$ and recall the notation $P_{\bf n}=(m_{\bf n},n_{\bf n}).$ We let $\sigma = \frac{m-m_{\bf n}}{n_{\bf n}-n}$ if $n_{\bf n}\neq n.$ Relative to $P_{\bf n}$, the source complex $P$ may belong to one of the three shaded regions in Figure \ref{Fig:relPos}.

\noindent {\em Region I.} Here $m>m_{\bf n}$ and $n<n_{\bf n}.$ Also, $\sigma > r_i$ and therefore  $\sigma \ge r_{i+1}$ and we have   
$$y(t_0) > (1/\delta') x(t_0)^{r_{i+1}}\ge(1/\delta)^{1/(n_{\bf n} - n)}x(t_0)^{\sigma},$$
\noindent which implies $\delta x(t_0)^{m_{\bf n}} y(t_0)^{n_{\bf n}}\ge x(t_0)^m y(t_0)^n.$

\noindent {\em Region II.} If $m\neq m_{\bf n}$ and  $n\neq n_{\bf n}$ we have $m>m_{\bf n},\ n>n_{\bf n}$ and $\sigma<0.$ Condition (P3) implies
$$y(t_0) < \delta' x(t_0)^{\sigma} < \delta^{1/(n-n_{\bf n})}x(t_0)^{\sigma},$$
which again implies $x(t_0)^m y(t_0)^n\le \delta x(t_0)^{m_{\bf n}} y(t_0)^{n_{\bf n}}.$ If $m = m_{\bf n}$ or $n = n_{\bf n}$ we need to show $\delta > y^{n - n_{\bf n}}$ or $\delta > x^{m - m_{\bf n}},$ which are immediate from (P5).

\noindent {\em Region III.} Here $m<m_{\bf n}$ and $n>n_{\bf n}.$ We have $\sigma < r_i,$ so  $\sigma \le r_{i-1}$ and 
$$y(t_0) < \delta' x(t_0)^{r_{i-1}} \le \delta^{1/(n-n_{\bf n})}x(t_0)^{\sigma},$$
\noindent therefore $\delta x(t_0)^{m_{\bf n}} y(t_0)^{n_{\bf n}} > x(t_0)^m y(t_0)^n.$

The cases when $c(t_0)$ lies on the other three polygonal lines $[B_1\ldots B_{f+1}]$, $[C_1\ldots C_{e+1}]$ and  $[D_1\ldots D_{f+1}]$ are similar. Finally, suppose that $c(t_0)$ is on one of the four remaining sides, $[A_{e+1}B_1]$, $[B_{f+1}C_1]$, $[C_{e+1}D_1]$ and $[D_{f+1}A_1],$ for instance $c(t_0)\in [A_{e+1}B_1].$  Then we must have $n > n_{\bf n}$ and  
$$y(t_0) < \delta' x(t_0)^{\sigma} < \delta ^{1/(n-n_{\bf n})}x(t_0)^{\sigma},$$      
or  $x(t_0)^m y(t_0)^n  < \delta x(t_0)^{m_{\bf n}}y(t_0)^{n_{\bf n}}.$
\end{proof}

\begin{remark}\label{remarca} 
Note that if the inclusion $\kappa(t)\in(\eta,1/\eta)^{\cal R}$ is not necessarily true for all $t\ge 0,$ but holds for $t=t_0$ then the inequality (\ref{mainIneq}) is true (at time $t_0$):
\begin{equation}
\left(\sum_{P\to P'\in {\cal R}}\kappa_{P\to P'}(t_0) (P'-P)c(t_0)^P\right)\cdot{\bf n}\geq 0
\end{equation}
\end{remark}

\textbf{Lower-endotactic networks.} A useful observation can be made from the proof of Theorem \ref{thm:main}. Namely, we need not require that the network is endotactic in order to conclude that a {\em bounded} trajectory $T(c_0)$ is persistent. Instead, it is enough that the network is merely {\em lower-endotactic}, meaning that it passes the parallel sweep test for vectors that point inside the closed positive quadrant.  Let ${\bf V}_L$ be the set of inward normal vectors to the sides of strictly negative slope of ${\cal SC(R)}$  (see Proposition \ref{prop:endo}). In other words, a network is lower-endotactic if condition (\ref{star}) is satisfied for all vectors ${\bf v}\in{\bf V}_L\cup\{\bf i, \bf j\}.$ Indeed, we may construct the invariant polygon $\cal P$ such that the square $[0, M]^2$ contains $T(c_0).$ Then, if $c(t)$ reaches $\cal P,$ it must be on the polygonal line $[D_{f+1}A_1A_2\ldots A_{e+1}B_1].$ and the proof of (\ref{mainIneq}) only requires that $\cal R$ is lower endotactic. We cast this result as a corollary to the proof of Theorem \ref{thm:main}.
\begin{corollary}\label{cor:boundedPers}
 Any bounded trajectory of a lower-endotactic $\kappa$-variable mass-action system with two species is persistent.
\end{corollary}

\subsection{Permanence}\label{subsection:perm} The proof of Theorem \ref{thm:main} allows us to say more about the dynamics of (\ref{varMassAct}). Indeed, without much additional effort we can prove the following stronger version of Theorem \ref{thm:main}:
\begin{theorem}\label{thm:perm}
Any two-species endotactic $\kappa$-variable reaction system is permanent.
\end{theorem}
\begin{proof} Suppose that for some small $\epsilon>0$ and for some $\alpha_0>0$ we can construct the polygon ${\cal P}(\alpha_0+\epsilon).$ Recall that the polygon ${\cal P}= {\cal P}(\alpha)$  depends continuously on $\alpha$ and the function $conv({\cal P}):(0,\alpha_0+\epsilon)\to \mathbb R$ is strictly decreasing, i.e.  $conv({\cal P}(\alpha))\supset conv({\cal P}(\alpha'))$ whenever $\alpha<\alpha'<\alpha_0+\epsilon.$ Let $T(c(0))$ denote a trajectory of (\ref{varMassAct}) with initial condition $c(0).$ We will show that the trajectory $T(c(0))$ eventually ends up in $conv({\cal P}(\alpha_0)).$ 

Since  $\bigcup_{\alpha\in(0,\alpha_0+\epsilon]}conv({\cal P}(\alpha))=\mathbb R_{>0}^2$ we may choose $0<\alpha_1<\alpha_0$ such that $c(0)\in \bigcup_{\alpha\in[\alpha_1,\alpha_0+\epsilon]} conv({\cal P}(\alpha)).$ 
Let  $\Phi:\bigcup_{\alpha\in[\alpha_1,\alpha_0+\epsilon]} {\cal P}(\alpha)\to [\alpha_1,\alpha_0+\epsilon]$ be defined as $\Phi (x,y)=\alpha$ if $(x,y)\in{\cal P}(\alpha).$ We need to show that $\Phi(c(t))\ge\alpha_0$ for large enough $t$. Suppose this was false; then, since $\Phi^{-1}[\alpha_0,\infty)=conv({\cal P}(\alpha_0))$ and $\Phi^{-1}[\alpha_1,\infty)=conv({\cal P}(\alpha_1))$ are both forward invariant sets for (\ref{varMassAct}), we must have that $\Phi(c(t))\in[\alpha_1,\alpha_0]$ for all $t\ge 0.$

Note that $\Phi$ is differentiable on its domain except at the points on the fractional index curves $y = x^\sigma$ for all $\sigma\in \{r_{.5},\ldots, r_{e+.5},s_{.5},\ldots, s_{f+.5}\}.$ Let ${\bf c}\in\mathbb R_{>0}^2$ be such a point, and let $\Phi_1$ and $\Phi_2$ denote the smooth functions that define $\Phi$ on the two sides of the fractional curve in a neighborhood of $\bf c.$ The subgradient  \cite[Definition 8.3]{Rocka} of $\Phi$ at $\bf c$ is 
\begin{equation}\label{subgrad}
\partial \Phi({\bf c})=\{a\nabla\Phi_1({\bf c})+(1-a) \nabla\Phi_2({\bf c})\mid a\in[0,1]\}
\end{equation}
\noindent and therefore $\Phi$ is strictly continuous \cite[Definition 9.1]{Rocka}. Note that (\ref{subgrad}) is also valid for points $\bf c$ that do not belong to a fractional index curve and in that case $\nabla \Phi_1({\bf c})=\nabla \Phi_2({\bf c})=\nabla \Phi({\bf c}).$ Since $c(t)$ is smooth, it follows that $\Phi\circ c(t)$ is also strictly continuous; in particular, a generalized mean value theorem \cite[Theorem 10.48]{Rocka} implies that for all $t>0$ there is $\tau_t\in[0,t]$ such that 
\begin{equation}\label{meanVal}
\Phi(c(t))-\Phi(c(0)) = s_t t  \text{ for some scalar } s_t\in \partial(\Phi\circ c)(\tau_t)
\end{equation}
The chain rule for subgradients \cite[Theorem 10.6]{Rocka}
\begin{equation}\label{chain}
\partial(\Phi\circ c)(t)\subseteq \{{\bf v}\cdot \dot c(t)\mid {\bf v} \in\partial \Phi(c(t))\}
\end{equation}
connects the subgradient of $\Phi\circ c$ to that of $\Phi.$ Note that the proof of Lemma \ref{staysInside} can be modified to show that for any compact set $K\subset\mathbb R_{>0}^2$ there exists $\zeta>0$ such that$\nabla \Phi_1(c(t))\cdot \dot c(t)>\zeta$ and $\nabla \Phi_2(c(t))\cdot \dot c(t)>\zeta$ for $c(t)\in K.$ This, together with the special form (\ref{subgrad}) of $\partial \Phi,$ and together with (\ref{chain}) shows that 
$$\inf _{t\ge 0}\partial(\Phi\circ c)(t)>\zeta.$$  

From ({\ref{meanVal}}) we obtain 
$$\Phi(c(t))>\Phi(0)+\zeta t \text{ for all }t > 0,$$
which contradicts the fact that $\Phi(c(t))\in[\alpha_1,\alpha_0]$ for all $t\ge 0.$
\end{proof} 

\noindent {\bf Remark} ({\em existence of positive equilibria for two-species autonomous endotatic mass-action systems}). The existence of a positive equilibrium has been shown for weakly reversible mass-action systems in \cite{chinese}.  For endotactic networks we may observe that if the rate constant function depends on both time and the phase point $c$, $\kappa = \kappa(t,c)\in(\eta,1/\eta)$ then the set of all possible trajectories is still the same as for $\kappa=\kappa(t),$ and our persistence and permanence results still hold. In particular, if  $\kappa=\kappa(c)\in(\eta,1/\eta)$
for all $c\in\mathbb R _{>0}^2$, then Theorem \ref{thm:perm} together with the Brouwer Fixed Point Theorem guarantees the existence of a positive equilibrium in $conv({\cal P}(\alpha_0)).$ 

\section{The Global Attractor Conjecture for three-species networks} 
As mentioned in the Introduction, the Global Attractor Conjecture concerns the asymptotic behavior of complex-balances mass-action systems and is of central importance in Chemical Reaction Network Theory.

\begin{definition}
A point $c_*\in \mathbb R_{\ge 0}^n$ of a mass-action system $({\cal R}, {\cal S}, {\cal C}, \kappa)$ is  called complex-balanced equilibrium if the aggregate flow at each complex of $\cal R$ is zero. More precisely, for each $P_0\in{\cal C}$ we have
$$\sum_{P\to P_0}\kappa_{P\to P_0}{c_*}^P=\sum_{P_0\to P}\kappa_{P_0\to P}{c_*}^{P_0}.$$
A complex-balanced system is a mass-action system that admits a strictly positive complex-balanced equilibrium. 
\end{definition}

A complex-balanced system is necessarily weakly reversible \cite{cinci}.  A very important class of complex-balanced mass-action systems is that of weakly reversible, {\em deficiency zero} systems \cite{cinci, cincijumate}. The deficiency of a network is given by $n-l-s,$
\noindent where $n,$ $l$ and $s$ denote the number of complexes, the number of linkage classes and the dimension of the stoichiometric subspace. Computing the deficiency of a network is straightforward, as opposed to checking the existence of a complex-balanced equilibrium. 

It is known that  a complex-balanced system admits a unique positive equilibrium $c_{\Sigma}$ in each stoichiometric compatibility class $\Sigma$ and this equilibrium is complex-balanced \cite{cinci}. Moreover, each such equilibrium  admits a strict Lyapunov function which guarantees that  $c_{\Sigma}$ is locally asymptotically  stable with respect to ${\Sigma}$ \cite{cinci, paispe}. 
Recall from the Introduction that the Global Attractor Conjecture states that $c_{\Sigma}$ is in fact {\em globally} asymptotically  stable with respect to ${\Sigma}.$

It is known that for complex-balanced mass-action systems, all trajectories with positive initial condition converge to the set of equilibria \cite{sase}. As a consequence, showing that the unique positive equilibrium $c_{\Sigma}$ in the stoichiometric compatibility class $\Sigma$ is globally attractive amounts to showing that no trajectories with positive initial conditions have $\omega$-limit points on the boundary of $\Sigma$. 

In particular, any condition that guarantees the non-existence of boundary equilibria for a stoichiometric compatibility class $\Sigma$ implies that $c_{\Sigma}$ is globally attractive. 
For instance, it was shown that a face $F$ of $\Sigma$ contains boundary equilibria only if the species corresponding to coordinates that vanish on $F$ form a {\em semilocking set (siphon in Petri nets terminology)} \cite{angeli, sapte, noua}. It follows that if no face of $\Sigma$ corresponds to a semilocking set then $\Sigma$ does not have boundary equilibria and $c_{\Sigma}$ is globally attractive. Algebraic methods for computing siphons have been devised in \cite{ShiuSturmfels}.

Some progress has been made for the remaining case where boundary equilibria cannot be ruled out. 
It was shown that $\omega$-limit points cannot be vertices of $\Sigma$ \cite{opt, sapte} and that codimension-one faces of $\Sigma$ are ``repelling" \cite{noua}. As a consequence, {\em  the Gobal Attractor Conjecture holds for systems with two-dimensional stoichiometric subspace} \cite{noua}. For systems with three-dimensional and higher-dimensional stoichiometric subspace the conjecture is  open. In particular, previous results do not apply to the general three-species systems.  In what follows we show the conjecture to be true in this case.

\begin{theorem}\label{gac3sp}
The Global Attractor Conjecture holds for three-species networks.
\end{theorem}
\begin{proof} 
Let $(\cal S,\cal C, \cal R, \kappa)$ denote a complex-balanced reaction system with ordered set of species $\{X, Y, Z\}$ and denote the corresponding concentration vector at time $t$ by $c(t)=(x(t), y(t), z(t)).$ Let $c_0\in\mathbb R_{>0}^3$ be arbitrary and let  $T(c_0)$ denote the forward trajectory of the concentration vector with initial condition $c_0.$  We want to prove that $T(c_0)$ converges to the unique positive equilibrium in the stoichiometric compatibility class of $c_0.$

The existence of the Lyapunov function guarantees, on one hand, that there is a neighborhood of the origin that is not visited by $T(c_0),$ and on the other hand, that $T(c_0)$ is bounded \cite{opt, sapte}. In other words, there exists $\epsilon>0$ such that $x(t)+y(t)+z(t)>3\epsilon$ and $x(t),\ y(t)$ and $z(t)$ are all smaller than $1/\epsilon$ at all times $t\ge 0.$

Our previous discussion explained that it is enough to rule out boundary $\omega$-limit points for $T(c_0).$
We will construct a compact set ${\cal K}\subset \mathbb R_{>0}^3$ such that $T(c_0)\subset \cal K.$

If $\pi_{xy}$ denotes the projection onto the coordinate coordinate plane $xy,$ then
\begin{equation}\label{projform}
\left[
\begin{matrix}
\dot x(t)\\
\dot y(t)
\end{matrix}
\right]
=\sum_{P\to P'\in{\cal R}}\kappa_{P\to P'}z(t)^{P_{Z}}(x(t), y(t))^{\pi_{xy}(P)}\pi_{xy}(P'-P).
\end{equation}
\noindent where $P_Z$ denotes the stoichiometric coefficient of the species $Z$ in $P$ and $\pi_{xy}(P)$ is the complex with species $\{X,Y\}$ obtained by removing species $Z$ from $P.$

Let 
$\kappa_{min} = \min\{\kappa, 1/\kappa\mid \kappa \text{ is a rate constant of some reaction in } \cal R\}$ 
and let $s_{max}$ denote the maximum stoichiometric coefficient of any species in a complex of $\cal R.$
Let
\begin{equation}\label{etagac}
\eta = \kappa_{min}\epsilon^{s_{max}}
\end{equation}
and note that $\eta<1.$

The reactions of $\cal R$ project onto reactions among complexes $\pi_{xy}(\cal C)$ with species $X$ and $Y.$ Let $\pi_{xy}(\cal R)$ denote the resulting reaction network. 
Motivated by (\ref{projform}) we consider the $\kappa$-variable mass-action system of variables $x$ and $y$ given by
\begin{equation}\label{projj}
\left[
\begin{matrix}
\dot x(t)\\
\dot y(t)
\end{matrix}
\right]
=\sum_{P\to P'\in{\pi_{xy}({\cal R})}}\kappa_{P\to P'}(t)(x(t), y(t))^{P}(P'-P).
\end{equation}
\noindent where $\kappa(t)\in(\eta, 1/\eta)$ for all $t\ge 0.$  
Since $\cal R$ is weakly reversible,  $\pi_{xy}(\cal R)$ is weakly reversible as well. In particular, $\pi_{xy}(\cal R)$ is endotactic and we may construct an invariant polygon ${\cal P}_{xy}$ for (\ref{projj}) as in the proof of Theorem \ref{thm:main}. Section \ref{sec:geom} contains the details of the construction and we use its notations in what follows. 

In Section \ref{sec:geom} the vertices of $\cal P$ lie on the ``fractional indices" curves to make ${\cal P}={\cal P}(\alpha)$ vary continuously with $\alpha.$ As noted in Remark \ref{rem:frac}, this constraint is not necessary for proving that ${\cal P}$ is an invariant set for the corresponding $\kappa$-variable mass-action kinetics. For the purpose of this section we will instead construct ${\cal P}={\cal P}_{xy}$ such that $A_1D_{f+1}$ is a vertical line and that its distance to the $y$ axis is equal to the distance from the horizontal line  $A_{e+1}B_1$ to the $x$ axis; let $d>0$ denote this distance. Since ${\cal P}_{xy}$ may be chosen to contain any compact subset of $\mathbb R_{\ge 0}^3$  (see (\ref{covers})), we assume that it satisfies (using the notations from Figure \ref{Fig:bigPicture}): 
\begin{equation}\label{conditie}
\text{the square } [\epsilon, 1/\epsilon]^2 \text{ is included in the square } [\xi, M]^2.
\end{equation}

We construct in the same way ${\cal P}_{yz}$  and ${\cal P}_{zx}$ and we define
\begin{equation}
{\cal K} = \{(x,y,z)\in [0, 1/\epsilon]^3 |\ (x,y)\in \text{conv}({\cal P}_{xy}), (y,z)\in \text{conv}({\cal P}_{yz}), (z,x)\in \text{conv}({\cal P}_{zx})\}.
\end{equation} 

Note that $\cal K$ is a compact subset of $\mathbb R_{>0}^3$ and $c_0\in\cal K.$ We show that the trajectory $T(c_0)$ is included in $\cal K.$ To this end, suppose that at some time point $t_0$ the trajectory reaches the boundary of $\cal K.$ It is then enough to check that
\begin{equation}\label{toprovegac}
{\bf n}\cdot \dot c(t_0)\ge 0
\end{equation}
where $-{\bf n}\in N_{\cal K}(c(t_0))$ is any generator of the normal cone of $\cal K$ at $c(t_0).$

\begin{figure}
\begin{center}
\begin{tikzpicture}[>=stealth', scale = 1] 

\draw[->,thin] (0,2.83)--(0,3.2);
\draw[->,thin] (2.83,0)--(3.2,0);
\draw[->,thin] (-1.47,-1.47)--(-1.65,-1.65);

\begin{scope}
\clip (.2,2.8)--(2.8,2.8)--(2.8,.2).. controls (2.79, -.05) .. (2.7,-.15)--(1.4,-1.45)--(-1.2,-1.45)  .. controls (-1.47, -1.47) .. (-1.55,-1.3)--(-1.55,1.25)--(-.15,2.65) .. controls (-.05,2.8) .. (.2,2.8);
\shadedraw[inner color = blue!100!green!13,outer color=white,draw=white] (-3,-3) rectangle +(7,7);
\end{scope}

\draw[color = gray!10, fill=blue!12!white,thin] (.2,.2)--(.2,2.8) .. controls (-.05,2.8).. (-.15,2.65)--(-.15,.1);
\draw[color = gray!10, fill=blue!12!white,thin] (2.8,.2) .. controls (2.79, -.05).. (2.7,-.15)--(.1,-.15)--(.2,.2);
\draw[color = gray!10, fill=blue!12!white,thin] (-1.2,-1.45) .. controls (-1.47, -1.47).. (-1.55,-1.3)--(-.15,.1)--(.1, -.15);
\draw[color = blue!35!white, thin] (.2,.2)--(.2,2.8);
\draw[color = blue!35!white,thin] (.2,.2)--(2.8,.2);
\draw[color = blue!35!white,thin] (.1,-.15)--(2.7,-.15);
\draw[color = blue!35!white,thin] (.1,-.15)--(-1.2,-1.45);
\draw[color = blue!35!white,thin] (-.15,.1)--(-.15,2.65);
\draw[color = blue!35!white,thin] (-.15,.1)--(-1.55,-1.3);
\draw[color = blue!35!white] (.2,2.8) .. controls (-.05,2.8) .. (-.15,2.65);
\draw[color = blue!35!white] (2.8,.2) .. controls (2.79, -.05) .. (2.7,-.15);
\draw[color = blue!35!white] (-1.2,-1.45) .. controls (-1.47, -1.47) .. (-1.55,-1.3);

\draw [color = blue!35!white,fill=blue!25!white](.2,.8)--(.8,.2)..controls(.92,.05)..(.65,-.15)--(-.25,-.5)..controls(-.65,-.6)..(-.6,-.35)--(-.15,.65)..controls (.02,.95)..(.2,.8);

\draw[color = gray!60!white,thin](.92,.2)--(.92,.95)--(.2,.95);
\fill (.92,.95) circle (.05);
\scriptsize
\node at (1.1,1.15) {$(d, 3\epsilon, 3\epsilon)$};

\draw[color = gray!60!white,thin](2.8,.2)--(2.8,2.8)--(.2,2.8);
\fill (2.8,2.8) circle (.05);
\scriptsize
\node at (3,3) {$(d, 1/\epsilon, 1/\epsilon)$};
\end{tikzpicture}
\end{center}
\caption{The subset of $\partial\cal K$ that can be reached by $T(c_0).$ The phase point $c(t_0)$ may belong to one of the three flat parts where one of the coordinate is $d$, or on one of the cylindrical parts along the axes.}\label{Fig:bdK}
\end{figure}
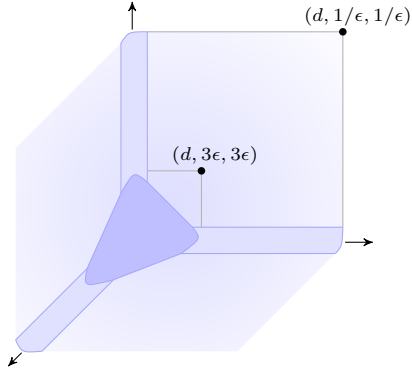

Let ${\cal L}_{xy}$ denote the subset of ${\cal P}_{xy},$ given by the polygonal line $D_{f+1}A_1A_2\ldots A_{e+1}B_1$ (see figure \ref{Fig:bigPicture}) and similarly define ${\cal L}_{yz}$ and ${\cal L}_{zx}.$ Because $c(t)\in(0,1/\epsilon)^3$ for all $t\ge 0,$ (\ref{conditie}) implies
\begin{equation}\label{eq:onboundary}
c(t_0)\in [0,1/\eta]^3 \cap \big(({\cal L}_{xy}\times \mathbb R_{>0})
\cup ({\cal L}_{yz}\times \mathbb R_{>0})\cup ({\cal L}_{zx}\times \mathbb R_{>0})\big)
\end{equation}

Since $x(t_0)+y(t_0)+z(t_0)>3\epsilon$, at least one of $x(t_0), y(t_0)$ and $z(t_0)$ is larger than $\epsilon$. Assume $z(t_0)>\epsilon$. Then (\ref{eq:onboundary}) implies  
that
\begin{equation}\label{onpxy}
c(t_0)\in {{\cal P}_{xy}}\times \mathbb R_{>0}, 
\end{equation}
and therefore the $z$ coordinate of $\bf n$ is zero. It follows that inequality (\ref{toprovegac}) is equivalent to 
\begin{equation}\label{gatagac}
\pi_{xy}({\bf n})\cdot \left(\sum_{P\to P'\in {\cal R}}\kappa_{P\to P'}z(t_0)^{P_Z}(x(t_0), y(t_0))^{\pi_{xy}(P)} (\pi_{xy}(P')-\pi_{xy}(P))\right)\ge 0.
\end{equation}
where $P_{Z}$ denotes the stoichiometric coefficient of the species $Z$ in $P$.
Note that $-\pi_{xy}({\bf n})$ is a generator of the normal cone of ${\cal P}_{xy}$ at $(x(t_0),y(t_0)).$ Our choice of $\eta$ (\ref{etagac}) implies that $\kappa_{P\to P'}z(t_0)^{P_Z}\in(\eta, 1/\eta)$ and (\ref{gatagac}) is implied by Remark \ref{remarca}. 
\end{proof}

\section{Examples}

\subsection{Thomas-type models} The Thomas mechanism (see \cite[ch. 6]{murray}) is a substrate inhibition model for a specific reaction involving oxygen and uric acid in the presence of the enzyme uricase. After nondimensionalization the ODEs for oxygen $(v)$ and  uric acid $(u)$ become
\begin{align}\label{thomas}
&\frac{du}{dt} = a-u-T(u,v)uv,\\
&\frac{dv}{dt} = \alpha(b-v)-T(u,v)uv, \nonumber
\end{align}
\noindent where in the Thomas model $T(u,v)=\rho(1+u+Ku^2)^{-1}.$ Here all parameters $a,b,\alpha,\rho$ and $K$ are positive. Using Theorem \ref{thm:perm} we can show that (\ref{thomas}) is permanent for any function $T(u,v)\geq 0$ that is continuous and does not vanish on a compact $K\supset[0,a]\times[0,b]$ (in particular, this is true for function $T$ in the Thomas mechanism). Indeed, for some time point $t_0$ we have $(u(t),v(t))\in K;$  continuity of $T$ guarantees that for $t\geq t_0$ we have $\eta < T(u(t),v(t))<1/\eta$ for some $\eta>0.$ The dynamical system (\ref{thomas}) can be written as a $\kappa$-variable mass-action system with reactions given in Figure \ref{Fig:examples} (a), where the reaction rates are specified on the reaction arrows and $T(t)=T(u(t),v(t)).$ This network is endotactic and Theorem \ref{thm:perm} implies that the dynamical system (\ref{thomas}) is permanent. 

\subsection{Power-law systems} In the proof of Theorem \ref{thm:main} we have not used the fact that the exponents $P$ of monomials $c^P$ in (\ref{varMassAct}) are nonnegative integers; this is the case when $P\subset \mathbb Z_{\ge 0}^2$ represents a chemical complex, which has been our framework thus far. In fact, the proof of Theorem \ref{thm:main} accommodates the case when $P\subset \mathbb R^2$ has any real components. 

More precisely, keeping the notation $c(t)=(x(t),y(t)),$ we consider an ODE system of the following form:
\begin{equation}\label{realPower}
\dot c(t) = \sum_{P\in {\cal {SC}}}\sum_{{\bf v}\in V(P)} \kappa_{P,{\bf v}}(t) c^{P}{\bf v}
\end{equation}
\noindent where ${\cal{ SC}}\subset \mathbb R^2$ is the set of  ``sources" $P$ and $V(P)$ is the set of ``reaction vectors" with source $P.$ We also suppose that $\kappa_{P,{\bf v}}(t)\in(\eta,1/\eta)$ for some $0<\eta<1.$   
With this interpretation, (\ref{realPower}) is a generalization of a $\kappa$-variable mass-action ODE system (\ref{varMassAct}) and Theorems \ref{thm:main} and \ref{thm:perm} still apply.

A particularly important example of power-law, not necessarily polynomial systems is the class of {\em S-systems} (\cite{Savageau}), where each component of the right hand side of (\ref{realPower}) consists of a difference of two ``generalized monomials" (i.e., monomials with real exponents). S-systems are common in the modeling of metabolic and genetic networks. For example, consider the following S-system:
\begin{align}\label{ex:S}
&\frac{dx}{dt} = 2x^{-1}y^{1.5} - y^{0.8}\\
&\frac{dy}{dt} = y^{-2} - \sqrt{5} x^{-1}y^{1.5}\nonumber
\end{align}
Note that it is not obvious that trajectories of (\ref{ex:S}) cannot reach $\partial \mathbb R_{\ge 0}^2$ in finite time. However, using Theorem \ref{thm:perm} we can easily see that (\ref{ex:S}) is in fact permanent, and, in particular, all trajectories with positive initial condition will be bounded away from zero and infinity. Indeed, the generalized monomials in (\ref{ex:S}), i.e. the points $(-1,1.5),(0,0.8)$ and $(0,-2),$ as well as the corresponding ``reaction vectors" $(2,-\sqrt{5}), (-1,0)$ and $(0,1)$ are illustrated in Figure \ref{Fig:examples} (b). This configuration is endotactic and Theorem \ref{thm:perm} applies.

Note that the same conclusion holds for any system of the form 
\begin{align}\label{ex:Svar}
&\frac{dx}{dt} = 2\kappa_1(t)x^{-1}y^{1.5} - \kappa_2(t)y^{0.8}\\
&\frac{dy}{dt} = \kappa_3(t)y^{-2} - \sqrt{5} \kappa_1(t)x^{-1}y^{1.5}\nonumber
\end{align}
\noindent where $\kappa_i(t)\in(\eta,1/\eta)$ for all $i\in\{1,2,3\}$ and for all $t\ge 0.$

%
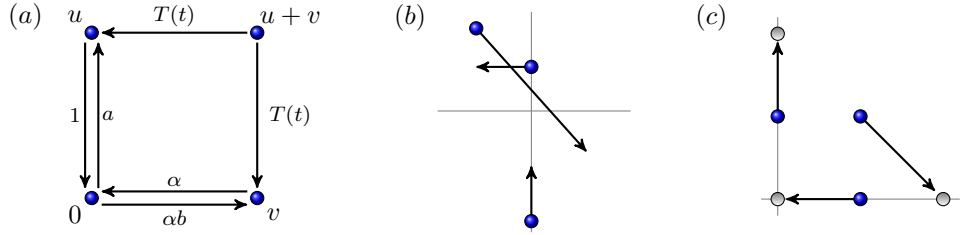
\begin{figure}
\begin{center}
\begin{tabular}{ccc}
\begin{tikzpicture}[>=stealth', scale = .22] 
\node at (-4,11) {$(a)$};
\node at (-1,-1) {$0$};
\node at (-1,11) {$u$};
\node at (12,11) {$u+v$};
\node at (11,-1) {$v$};

\footnotesize
\node at (5, -1.2) {${\alpha b}$};
\node at (5, 1) {$\alpha$};
\node at (1,5) {$a$};
\node at (-1,5) {$1$};
\node at (12,5) {$T(t)$};
\node at (5,11) {$T(t)$};

\draw[->,thick] (.6,-.4)--(9.4,-.4);
\draw[<-,thick] (.6,.4)--(9.4,.4);
\draw[->,thick] (.4,.6)--(.4,9.4);
\draw[<-,thick] (-.4,.6)--(-.4,9.4);
\draw[<-,thick] (10,.6)--(10,9.4);
\draw[<-,thick] (.6,10)--(9.4,10);

\shadedraw [shading=ball] (0,0) circle (.4cm);
\shadedraw [shading=ball] (0,10) circle (.4cm); 
\shadedraw [shading=ball] (10,0) circle (.4cm);
\shadedraw [shading=ball] (10,10) circle (.4cm); 
\end{tikzpicture} 
&
\hspace{.2cm}
\begin{tikzpicture}[>=stealth', scale = .733] 
\node at (-2.2,1.7) {$(b)$};
\draw[gray](0,1.8)--(0,-2.2); 
\draw[gray](-1.7,0)--(1.8,0);
\draw[->, thick] (0,.8)--(-1,.8);
\draw[->,thick] (0,-2)--(0,-1);
\draw[->,thick] (-1,1.5)--(1,-0.736);

\shadedraw [shading=ball] (-1,1.5) circle (.12cm);
\shadedraw [shading=ball] (0,.8) circle (.12cm); 
\shadedraw [shading=ball] (0,-2) circle (.12cm);
\end{tikzpicture} 
\hspace{.2cm}
&
\begin{tikzpicture}[>=stealth', scale = .22] 
\node at (-4,11) {$(c)$};
\draw[gray](0,11)--(0,-1);
\draw[gray](-1,0)--(11,0);
\draw[->,thick] (0,5)--(0,9.5);
\draw[->,thick] (5,0)--(0.5,0);
\draw[->,thick] (5,5)--(9.5,.5);
\shadedraw [] (0,0) circle (.4cm);
\shadedraw [] (0,10) circle (.4cm); 
\shadedraw [] (10,0) circle (.4cm);
\shadedraw [shading=ball] (0,5) circle (.4cm); 
\shadedraw [shading=ball] (5,0) circle (.4cm); 
\shadedraw [shading=ball] (5,5) circle (.4cm); 
\node at (11,-1.5) {$$};
\end{tikzpicture} 
\end{tabular}
\end{center}
\caption{Examples of reaction networks: (a)--Thomas model (\ref{thomas}); (b)--the S-system (\ref{ex:S}); (c)--Lotka-Volterra system (\ref{ex:LV}).}\label{Fig:examples}
\end{figure}
\subsection{Lotka-Volterra systems} The classical two-species predator-prey model 
\begin{equation}\label{ex:LV}
A\to 2A\quad A+B\to 2B\quad B\to 0
\end{equation}
is neither endotactic nor lower endotactic. Since, {\em for fixed parameters} its trajectories are either constant or closed orbits, the system is not permanent (see Figure \ref{Fig:examples} (c)). On the other hand, the fixed-parameter system has bounded trajectories and is persistent, which seems to suggest that the requirement of entotactic network may be weakened in  Corollary \ref{cor:boundedPers}. Note, however, that the result of the corollary concerns $\kappa$-variable mass-action systems. It is not hard to show that, in general, the  {\em $\kappa$-variable }  Lotka-Volterra system (\ref{ex:LV}) is {\em not} persistent.

\subsection{Examples for the three-species Global Attractor Conjecture} 
We present two examples of reaction networks for which no previously known results 
can resolve the question of global asymptotic stability, but for which Theorem \ref{gac3sp} applies. 
\begin{equation}\label{ex:gac}
\begin{tikzpicture}[very thin, scale = 1.1] 
\node at (-6.7,.85) {$(a)\ A\rightleftharpoons B\rightleftharpoons A+B\rightleftharpoons A+C$};
\node at (-2.3,.85) {$(b)\ A+B\rightleftharpoons A+C$};
\node at (0,-.05) {$2A$};
\node at (-.5,.88) {$A$};
\node at (.5,.88) {$B$};
\draw[semithick,->] (-.35,.88)--(.35,.88);
\draw[semithick,->] (-0.1,0.1)--(-.43,.7);
\draw[semithick ,<-] (0.1,0.1)--(.43,.7);
\end{tikzpicture} 
\end{equation}
The first example, given by (\ref{ex:gac}) (a) is Example 5.4 from \cite{noua}; it is reversible with one linkage class. On the other hand, example (\ref{ex:gac}) (b) is weakly-reversible, but not reversible and has two linkage classes. Moreover, the deficiency of both networks is zero and therefore the two mass-action systems are complex-balanced.

In both examples all stoichiometric compatibility classes are equal to $\mathbb R_{\ge 0}^3.$ If $(x(t),y(t),z(t))$ denotes the concentrations of $(A, B, C),$ then it is easy to see that all points on the nonnegative $z$ axis are equilibria for both examples. Therefore boundary equilibria exist, and moreover, except for the origin, they all lie on a codimension-two face of $\mathbb R_{\ge0}^3.$ In this case all previously  known results stay silent, but Theorem \ref{gac3sp} guarantees that, in each example, the unique positive equilibrium is globally asymptotically stable.



\begin{thebibliography}{99}

\bibitem{sapte}
D.F. Anderson, 
{\em Global asymptotic stability for a class of nonlinear chemical equations,} 
{SIAM J. Appl. Math}, 68:5, 1464--1476, 2008. 

\bibitem{noua}
D.F. Anderson, A. Shiu,
{\em The dynamics of weakly reversible population processes near facets,} to appear in SIAM J. Appl. Math, {\tt arXiv:0903.0901v4.}

\bibitem{angeli}
D.Angeli, P. De Leenheer, and E. Sontag, 
{\em A Petri net approach to persistence 
analysis in chemical reaction networks},
In I. Queinnec, S. Tarbouriech, G. Garcia, and S-I. 
Niculescu, editors, Biology and Control Theory: Current Challenges (Lecture Notes in Control 
and Information Sciences Volume 357), 181--216. Springer-Verlag, Berlin, 2007. 

\bibitem{Blanchini}
F. Blanchini, 
{\em Set invariance in control,}
Automatica 35, 1747--1767, 1999.

\bibitem{opt}
G. Craciun, A. Dickenstein, A. Shiu, B. Sturmfels, 
{\em Toric Dynamical Systems, }
{J. Symb. Comp.}, 44:11, 1551--1565,  2009.

\bibitem{craciun_feinberg_PNAS}
G.~Craciun, Y.~Tang, M.~Feinberg,
{\em Understanding bistability in complex enzyme-driven reaction networks,} 
Proc. Natl. Acad. Sci. 103:23, 8697--8702, 2006.

\bibitem{chinese}
J. Deng, M. Feinberg, C. Jones, A. Nachman,
{\em On the steady states of weakly reversible chemical reaction networks},
preprint.

\bibitem{cinci}
M. Feinberg,
{\em Lectures on chemical reaction networks.} 
Notes of lectures given at the Mathematics Research Center of the University of Wisconsin in 1979, 
{\tt http://www.che.eng.ohio-state.edu/$\sim$FEINBERG/LecturesOnReactionNetworks.} 

\bibitem{cincijumate}
M. Feinberg,
{\em Chemical Reaction Networks structure and the stability of complex isothermal reactors -- 
I. The deficiency zero and deficiency one theorems.} 
{Chem. Eng. Sci.} 42:10, 2229-2268, 1987.

\bibitem{Freedman}
H.I. Freedman, P. Waltman, 
{\em Mathematical analysis of some three species food chain models,} 
{Math. Biosci.}, 33, 257--273, 1977.

\bibitem{patru}
F. Horn, 
{\em The dynamics of open reaction systems,} 
Mathematical aspects of chemical and biochemical problems and quantum chemistry, 
SIAM-AMS Proceedings, Vol. VIII, 125--137, 1974. 

\bibitem{paispe}
F. Horn, R. Jackson,
{\em General mass action kinetics, }
Archive for Rational Mechanics and Analysis, 47,  81--116, 1972.

\bibitem{murray}
J.D. Murray, {\em Mathematical Biology: I. An Introduction,}
Third edition, Springer, 2002.

\bibitem{nagumo}
M. Nagumo, 
{\em \"Uber die Lage der Integralkurven gew\"{o}hnlicher Differentialgleichungen,} 
Proceedings of the Physico-Mathematical Society of Japan, 24 (1942), 551--559. 

\bibitem{Savageau}
 M.A. Savageau,
 {\em Biochemical systems analysis. I. Some mathematical properties of the rate law for the component enzymatic reactions,}
Journal of  Theoretical  Biology 25(3), 365--369, 1969.

\bibitem{Schuster}
P. Schuster, K. Sigmund, R. Wolf
{\em On $\omega$-limits for competition between three species,}
{SIAM J. Appl. Math.} 37, 49--55, 1979.

\bibitem{ShiuSturmfels}
A. Shiu, B. Sturmfels,
{\em Siphons in chemical reaction networks}, 
Bulletin of Mathematical Biology, 72:6, 1448--1463 (2010)

\bibitem{sase}
E. Sontag, 
{\em Structure and stability of certain chemical networks 
and applications to the kinetic proofreading model of T-cell receptor signal 
transduction,} 
{IEEE Trans. Automat. Control} 46, 1028--1047, 2001. 

\bibitem{LV}
Y. Takeuchi,
{\em Global Dynamical Properties of Lotka-Volterra Systems}
World Scientific Publishing, 1996.

\bibitem{Rocka}
R. T. Rockafellar, R. J-B Wets,
{\em Variational Analysis, (Grundlehren der mathematischen Wissenschaften, volume 317),}
Springer, 1998.

\bibitem{Sontag1}
Eduardo Sontag,
{\em Structure and stability of certain chemical networks and applications to the 
kinetic proofreading model of T-cell receptor signal transduction,} 
IEEE Trans. Automat. Control, 46, 1028--1047, 2001. 


\end{thebibliography}
\end{document}